\documentclass[11pt,reqno]{amsart}
\usepackage{graphicx}

\usepackage{amssymb}
\usepackage{cite}
\usepackage{amsmath}
\usepackage{latexsym}
\usepackage{amscd}
\usepackage{amsthm}
\usepackage{mathrsfs}
\usepackage{url}
\usepackage{enumitem}

\usepackage{tikz} 

\usepackage[backref=page,linktocpage=true,colorlinks,citecolor=magenta,linkcolor=blue,urlcolor=magenta]{hyperref}

\makeatletter
\@namedef{subjclassname@2020}{%
  \textup{2020} Mathematics Subject Classification}
\makeatother

\numberwithin{equation}{section}
\allowdisplaybreaks[4]

\newtheorem{thm}{Theorem}
\newtheorem{corr}[thm]{Corollary}
\newtheorem{lem}[thm]{Lemma}
\newtheorem{prop}[thm]{Proposition}
\newtheorem{conj}[thm]{Conjecture}
\newtheorem{ques}[thm]{Question}

\theoremstyle{definition}

\theoremstyle{remark}
\newtheorem{rem}[thm]{Remark}
\newtheorem*{ack}{Acknowledgment}
\def\R{\mathbb R}

\def\H{\mathbb H}
\def\SS{\mathbb S}

\def\pt{\partial}

\begin{document}
\title[Scaling inequalities for Steklov eigenvalues]{Scaling inequalities for Steklov eigenvalues in space forms and sharp eigenvalue estimates on warped product manifolds}
\author[Z.~Lv]{Zongyi~Lv}
\address{School of Mathematics, Sichuan University, Chengdu 610065, Sichuan,  P.~R.~China}
\email{\href{mailto:15828971863@163.com}{15828971863@163.com}}
\author[C.~Xiong]{Changwei~Xiong}
\address{School of Mathematics, Sichuan University, Chengdu 610065, Sichuan,  P.~R.~China}
\email{\href{mailto:changwei.xiong@scu.edu.cn}{changwei.xiong@scu.edu.cn}}
\author[Y.~Zou]{Yuxun~Zou}
\address[a]{School of Mathematical Sciences, University of Chinese Academy of Sciences, Beijing 100049, P.~R.~China;}
\address[b]{Institute of Applied Mathematics, Academy of Mathematics and Systems Science, Chinese Academy of Sciences, Beijing 100190, P.~R.~China}
\email{\href{mailto:zouyuxun25@mails.ucas.ac.cn}{zouyuxun25@mails.ucas.ac.cn}}

\date{\today}

\maketitle

\begin{abstract}
In the first part, we derive monotonicity of the normalized spectra for the second-order Steklov problem and two fourth-order Steklov problems on the $2$-dimensional geodesic disks with respect to the geodesic radius in the sphere and the hyperbolic space. The normalizations are made using four natural geometric factors. As corollaries, we get Escobar-type bounds for Steklov eigenvalues on $2$-dimensional geodesic disks with varying curvature in space forms. We also get two monotonicity results for higher-dimensional cases. In the second part, we obtain some sharp bounds concerning the spectra of the two fourth-order Steklov problems on warped product manifolds with non-negative Ricci curvature and a strictly convex boundary. In particular, we confirm Qiaoling Wang and Changyu Xia's conjecture (2018) on the sharp lower bound of the first non-zero eigenvalue of a fourth-order Steklov problem in the case of $3$-dimensional warped product manifolds.
\end{abstract}

\section{Introduction}\label{sec1}

One of the most extensively-investigated topics in the spectral geometry is the estimates on various eigenvalues. In the first part of this work we are concerned with the scaling properties of the Steklov eigenvalues, which may lead to estimates on eigenvalues. The scaling properties of eigenvalues in the Euclidean space are well-known in the literature. For example, for a bounded Euclidean domain $M\subset \R^n$, all its Dirichlet eigenvalues $\lambda(M)$ satisfy
\begin{align*}
t^2\lambda(tM)=\lambda(M),\quad t>0.
\end{align*}
In comparison, this kind of scaling properties is less known in the curved spaces such as the spherical and hyperbolic spaces. In \cite{LL23b} Langford and Laugesen studied the scaling properties for Dirichlet and Neumann eigenvalues in the spherical and hyperbolic spaces, mainly in the $2$-dimensional cases. Here we intend to investigate the scaling properties for three kinds of Steklov eigenvalues in these curved spaces. For $2$-dimensional cases, we can get a quite complete description; and interestingly for certain higher-dimensional cases we are also able to get monotonicity results. In the second part of this work, we consider the more general context of warped product manifolds and derive some sharp results of the fourth-order Steklov eigenvalues on such manifolds with non-negative Ricci curvature and a strictly convex boundary.

The motivation of our investigation in this work includes not only the scaling properties of Steklov eigenvalues in curved spaces, but also fine estimates of Steklov eigenvalues in general Riemannian manifolds. In other words, we will obtain estimates of eigenvalues on general Riemannian manifolds with boundary, provided that comparison between them and those on model spaces such as the warped product manifolds with rotationally symmetric metrics can be established (e.g., the comparison results \cite[Thms.~1 and 2]{Esc00}), and meanwhile nice estimates of eigenvalues on these model spaces are derived. Therefore, to estimate eigenvalues properly on model spaces is one of the main steps to estimate them on general Riemannian manifolds. In fact, for well-known Dirichlet and Neumann eigenvalues, considerable interest has been devoted to getting satisfactory estimates for them on manifolds with rotationally symmetric metrics; see, e.g., \cite{Har25,Ber23,Kri22,BF17,Bag90,Bag91,BCG83,Art16,BB06,Cam94,DGM81,FH76,Gag80,Gri99,HMP16,Kat03,MT82,McK70,Pin81,Pin94,Pin78,Pin79,Sat82,Sav09,Wan95} and references therein. Hence it is natural to carry out similar studies for the Steklov eigenvalues.

Let $(M^n,g)$ be an $n$-dimensional smooth compact Riemannian manifold with boundary $\pt M$. We consider the following three types of Steklov eigenvalue problems
\begin{equation}\label{problem1}
	\begin{cases}
		\Delta u =0,&\text{ in }  M, \\
		\dfrac{\partial u }{\partial \nu }=\sigma u, &\text{ on }  \partial M,
	\end{cases}
\end{equation}
\begin{equation}\label{problem2}
	\begin{cases}
		\Delta ^{2}u =0,&\text{ in }  M, \\
		\dfrac{\partial u }{\partial \nu }=0, \ \dfrac{\partial \left ( \Delta u  \right )}{\partial \nu }+\xi u =0,&\text{ on }  \partial M,
	\end{cases}
\end{equation}
and
\begin{equation}\label{problem3}
	\begin{cases}
		\Delta ^{2} u =0,&\text{ in }  M, \\
		u =0, \ \Delta u =\eta \dfrac{\partial u }{\partial \nu },&\text{ on }  \partial M,
	\end{cases}
\end{equation}
where $\nu$ denotes the outward unit normal to $\partial M$. The problem \eqref{problem1} was introduced by Steklov \cite{Ste02} around 1900; see \cite{CGGS24,GP17,KKK14} for nice surveys and historical remarks on the second-order Steklov problem. The problems \eqref{problem2} and \eqref{problem3} were studied first by J.~R.~Kuttler and V.~G.~Sigillito \cite{KS68} in 1968 and by L.~E.~Payne \cite{Pay70} in 1970; see, e.g., \cite{BK22} for introduction on generalized fourth-order Steklov eigenvalue problems. See also Section~\ref{subsec2.3} for some introduction on these Steklov problems. In particular, we use $\sigma_m$, $\xi_m$ and $\eta_m$ starting from $m=0$ to denote the eigenvalues. Thus $\sigma_0=0$ and $\xi_0=0$. Moreover, we use $\sigma_{(m)}$, $\xi_{(m)}$ and $\eta_{(m)}$ to denote the eigenvalues without counting multiplicity.

In this work we are mainly concerned with the warped product Riemannian manifold $M^n=[0,R]\times \SS^{n-1}$ equipped with the warped product metric
\begin{align*}
g=dr^2+h^2(r)g_{\SS^{n-1}}.
\end{align*}
Throughout this paper we impose the following Assumption (A) on the warping factor $h(r)$.
\begin{itemize}
  \item[(A)] $h\in C^\infty([0,R])$, $h(r)>0$ for $r\in (0,R]$, $h'(0)=1$ and $h^{(2k)}(0)=0$ for all integers $k\geq 0$.
\end{itemize}
In particular, when $h(r)=\sin r$, $M$ is in the sphere; when $h(r)=\sinh r$, $M$ is in the hyperbolic space. See Section~\ref{subsec2.1} for more details on the setting. Besides, by the results in Section~\ref{subsec2.3}, the Steklov eigenvalues on warped product manifolds generally admit multiplicity. Thus we only need to study the eigenvalues $\sigma_{(m)}$, $\xi_{(m)}$ and $\eta_{(m)}$ without counting multiplicity.

Next we discuss the geometric factors for scaling in the sphere and the hyperbolic space. First note that the Steklov problems above admit the following scaling properties:
\begin{align*}
\sigma_k((M,c^2g))&=c^{-1}\sigma_k((M,g)), \ c>0,\\
\xi_k((M,c^2g))&=c^{-3}\xi_k((M,g)), \ c>0,\\
\eta_k((M,c^2g))&=c^{-1}\eta_k((M,g)), \ c>0.
\end{align*}
For the $2$-dimensional geodesic disk in the sphere ($h(r)=\sin r$), there are four natural geometric factors (see Figure~\ref{Fig1}): the geodesic radius (or the aperture) $R$, the Euclidean radius of the boundary circle $\sin R$, the stereographic radius $\tan (R/2)$, and the area $4\pi \sin^2 (R/2)$. Therefore in view of the scaling properties above it is natural to consider the monotonicity with respect to $R$ of the following quantities
\begin{align*}
\sigma_{(m)} R,\quad \sigma_{(m)} \sin R,\quad \sigma_{(m)} \tan (R/2),\quad \sigma_{(m)} \sin (R/2),\\
\xi_{(m)} R^3,\quad \xi_{(m)} \sin^3 R,\quad \xi_{(m)} \tan^3 (R/2),\quad \xi_{(m)} \sin^3 (R/2),\\
\eta_{(m)} R,\quad \eta_{(m)} \sin R,\quad \eta_{(m)} \tan (R/2),\quad \eta_{(m)} \sin (R/2).
\end{align*}
\begin{figure}
\centering
\begin{tikzpicture}[scale=3.2]

 \draw [<-] (0,1.3)node [right] {$z$}--(0,-1.2) ;
 \draw [->] (-1.2,0)--(1.2,0)node [right] {$x$};
 \node [above left] at (0,0) {$0$};
 \draw [dashed] (0,0)--(0.6,0.8);
 \draw [-] (-0.6,0.8) to [out=-20, in=200] (0.6,0.8);
 \draw [dashed] (-0.6,0.8) to [out=20, in=160] (0.6,0.8);
 \node [above] at (0.3,0.954) {$R$};
 \draw [dashed] (0,0.8) --(0.6,0.8);
 \draw [<-] (0.25,0.83)--(0.5,1.0) node [above right] {$\sin R$};
 \draw [dashed] (0,-1)--(0.6,0.8);
 \draw [<-] (0.16,-0.05) -- (0.4,-0.3) node [right] {$\tan (R/2)$};
 \draw (0,0.15) node [above right] {$R$} to [out=0, in=145] (0.09,0.12);
 \node [above left] at (0,1) {$\mathrm{Area}=4\pi \sin^2 (R/2)$};
 \node [above right] at (1,0) {$1$};
 \node [above right] at (0,1) {$N$};
 \node [below right] at (0,-1) {$S$};

  \draw (0,0) circle (1);
\end{tikzpicture} 
\caption{Geometric factors in $\SS^2$}
\label{Fig1}
\end{figure}
In the hyperbolic space ($h(r)=\sinh r$), we may consider similar problems concerning natural geometric factors (see Section~\ref{subsec2.2}).

In the remaining part of Section~\ref{sec1}, we will present and discuss our main results. In Sections \ref{subsec1.1}--\ref{subsec1.3} we shall consider the $2$-dimensional space forms, give a complete description of the scaling properties of the Steklov eigenvalues, and as corollaries present Escobar-type monotonicity results for these Steklov eigenvalues on geodesic disks with varying curvature. Then in Section~\ref{subsec1.4} we move to the higher-dimensional space forms and give two monotonicity results. Last in Section~\ref{subsec1.5} we discuss some sharp estimates on warped product manifolds.
\begin{rem}
Analogous results for Dirichlet and Neumann eigenvalues to those in Sections~\ref{subsec1.1}--\ref{subsec1.4} were derived by Langford and Laugesen in \cite{LL23b}. More precisely, they studied on $2$-dimensional geodesic disks in space forms the monotonicity of the full Dirichlet and Neumann spectra normalized by the stereographic radius and the Euclidean radius of the boundary circle, the second Neumann eigenvalue normalized by the area, the first and second Dirichlet eigenvalues normalized by the geodesic radius. As corollaries, they obtained Bandle-type bounds for Neumann eigenvalues \cite{Ban72} \cite[Cor.~3.9]{Ban80} and Cheng-type bounds for Dirichlet eigenvalues \cite{Che75} on geodesic disks with varying curvature. More recently, Harman \cite{Har25} studied scaling inequalities for Robin and Dirichlet eigenvalues.
\end{rem}
\begin{rem}
For the results in Section~\ref{subsec1.5}, the original motivation to get them is to study problems similar to the Escobar's conjecture \cite[p.~115]{Esc99}. See \cite{Esc97,Xio22,Xio21,XX24,Xio25} for the detailed motivation and related open problems.
\end{rem}


\subsection{Scaling inequalities in the $2$-dimensional sphere}\label{subsec1.1}

For the $2$-dimensional sphere, we have nice expressions for $\sigma_{(m)}$, $\xi_{(m)}$ and $\eta_{(m)}$. Precisely, we have (see, e.g., \cite[Prop.~4]{Esc00}\cite[(1.3)]{Xio21})
\begin{align*}
\sigma_{(m)}&=\frac{m}{\sin R}.
\end{align*}
And we have $\xi_{(m)}$ and $\eta_{(m)}$ as in Lemma~\ref{lem2} below in the sphere. Using these explicit expressions, we are able to get a quite complete description on the scaling inequalities in the $2$-dimensional sphere.

First consider the second-order Steklov problem. Since our main purpose of this work is to estimate eigenvalues and for the $2$-dimensional case we have already known the explicit expression for $\sigma_{(m)}=m/\sin R$, there is no need to derive the monotonicity properties with respect to $R$ of $\sigma_{(m)}R$, $\sigma_{(m)}\sin{R}$, $\sigma_{(m)}\tan{(R/2)}$, and $\sigma_{(m)}\sin{(R/2)}$. Just for comparison purpose it is straightforward to check that on $(0,\pi)$ the function $\sigma_{(m)}R$ is strictly increasing, $\sigma_{(m)}\sin{R}$ is constant, $\sigma_{(m)}\tan{(R/2)}$ is strictly increasing, and $\sigma_{(m)}\sin{(R/2)}$ is also strictly increasing.


The non-trivial cases are for the two fourth-order Steklov eigenvalue problems. Next we give results for $\xi_{(m)}$ of the type one fourth-order Steklov eigenvalue problem \eqref{problem2}.
\begin{thm}\label{thm2}
For the eigenvalue $\xi_{(m)}$ $(m \geq 1)$ of the geodesic disk of radius $R$ in the $2$-dimensional sphere, we have the following conclusions.
\begin{enumerate}
  \item The function $\xi_{(m)}R^{3}$ is strictly increasing on $(0,\pi)$.
  \item The function $\xi_{(m)}\sin^3{R}$ is strictly decreasing on $(0,\pi)$.
  \item The function $\xi_{(m)}\tan^3{(R/2)}$ is strictly increasing on $(0,\pi)$.
  \item The function $\xi_{(m)}\sin^3{(R/2)}$ is strictly increasing on $(0,\pi)$.
\end{enumerate}
\end{thm}
\begin{rem}
As $R\to 0+$, the function $\xi_{(m)}R^3$ converges to the $m$th Steklov eigenvalue on the Euclidean unit disk, $2m^2(m+1)$. Consequently, by (2) and (4) of Theorem~\ref{thm2}, we get
\begin{align*}
\xi_{(m)}\sin^3 R<2m^2(m+1)<8\xi_{(m)}\sin^3 \frac{R}{2},\quad R\in (0,\pi),
\end{align*}
or
\begin{align*}
\frac{m^2(m+1)}{4\sin^3 \frac{R}{2}}<\xi_{(m)}<\frac{2m^2(m+1)}{\sin^3 R},\quad R\in (0,\pi).
\end{align*}
Thus we get a two-sided estimate for $\xi_{(m)}$. Similar arguments apply to Theorems~\ref{thm3}, \ref{thm5} and \ref{thm6}.
\end{rem}

Last we get results for $\eta_{(m)}$ of the type two fourth-order Steklov eigenvalue problem \eqref{problem3}.
\begin{thm}\label{thm3}
For the eigenvalue $\eta_{(m)}$ $(m \geq 1)$ of the geodesic disk of radius $R$ in the $2$-dimensional sphere, we have the following conclusions.
\begin{enumerate}
  \item The function $\eta_{(m)}R$ is strictly increasing on $(0,\pi)$.
  \item The function $\eta_{(m)}\sin{R}$ is strictly decreasing on $(0,\pi)$.
  \item The function $\eta_{(m)}\tan{(R/2)}$ is strictly increasing on $(0,\pi)$.
  \item For $m=1$, there exists an $R_1\in (0,\pi)$, such that the function $\eta_{(1)} \sin (R/2)$ is strictly decreasing on $(0,R_1)$ and strictly increasing on $(R_1,\pi)$; while for $m\geq 2$, the function $\eta_{(m)}\sin{(R/2)}$ is strictly increasing on $(0,\pi)$.
\end{enumerate}
\end{thm}
\begin{rem}
For $m=0$, by Lemma ~\ref{lem2} below we see that $\eta_{(0)}=\cot(R/2)$. Then it is straightforward to check that on $(0,\pi)$ the function $\eta_{(0)}R$ is strictly decreasing, $\eta_{(0)}\sin{R}$ is strictly decreasing, $\eta_{(0)}\tan{(R/2)}$ is constant, and $\eta_{(0)}\sin{(R/2)}$ is strictly decreasing.
\end{rem}

\subsection{Scaling inequalities in the $2$-dimensional hyperbolic space}\label{subsec1.2}

As in the sphere, for the second-order Steklov eigenvalue problem in the $2$-dimensional hyperbolic space we have explicitly (see, e.g., \cite[Prop.~4]{Esc00}\cite[(1.3)]{Xio21})
\begin{align*}
\sigma_{(m)}=\frac{m}{\sinh R},
\end{align*}
and for the fourth-order problems we have explicit expressions for $\xi_{(m)}$ and $\eta_{(m)}$ as in Lemma~\ref{lem3} below. Then we may check directly that on $(0,+\infty)$ the function $\sigma_{(m)}R$ is strictly decreasing, the function $\sigma_{(m)}\sinh{R}$ is constant, the function $\sigma_{(m)}\tanh(R/2)$ is strictly decreasing, and the function $\sigma_{(m)}\sinh(R/2)$ is strictly decreasing.

Next, we turn to the non-trivial problems, the two fourth-order eigenvalue problems. First we have the following result for the problem \eqref{problem2}.
\begin{thm}\label{thm5}
For the eigenvalue $\xi_{(m)}$ $(m \geq 1)$  of the geodesic disk of radius $R$ in the $2$-dimensional hyperbolic space, we have the following conclusions.
\begin{enumerate}
  \item The function $\xi_{(m)}R^3$ is strictly decreasing on $(0,+\infty)$.
  \item The function $\xi_{(m)}\sinh^3R$ is strictly increasing on $(0,+\infty)$.
  \item The function $\xi_{(m)}\tanh^3(R/2)$ is strictly decreasing on $(0,+\infty)$.
  \item The function $\xi_{(m)}\sinh^3(R/2)$ is strictly decreasing on $(0,+\infty)$.
\end{enumerate}
\end{thm}

Second we have the result for the problem \eqref{problem3}.
\begin{thm}\label{thm6}
For the eigenvalue $\eta_{(m)}$ $(m \geq 1)$ of the geodesic disk of radius $R$ in the $2$-dimensional hyperbolic space, we have the following conclusions.
\begin{enumerate}
  \item When $m=1$, the function $\eta_{(m)}R$ is strictly increasing on $(0,+\infty)$; while for $m\geq 2$, there exists a positive constant $R_m$, which depends only on $m$, such that the function $\eta_{(m)} R$ is strictly decreasing on $(0, R_m)$, and strictly increasing on $(R_m,+\infty)$.
  \item The function $\eta_{(m)}\sinh R$ is strictly increasing on $(0,+\infty)$.
  \item The function $\eta_{(m)}\tanh(R/2)$ is strictly decreasing on $(0,+\infty)$.
  \item When $m \leq 2$, the function $\eta_{(m)}\sinh(R/2)$ is strictly increasing on $(0,+\infty)$; while for $m\geq 3$, there exists a positive constant $\bar{R}_m$, which depends only on $m$, such that the function $\eta_{(m)}\sinh(R/2)$ is strictly decreasing on $(0, \bar{R}_m)$, and strictly increasing on $(\bar{R}_m,+\infty)$.
\end{enumerate}
\end{thm}
\begin{rem}
For $m=0$, by Lemma ~\ref{lem3} below we see that $\eta_{(0)}=\coth(R/2)$. Then it is straightforward to check that on $(0,+\infty)$ the function $\eta_{(0)}R$ is strictly increasing, $\eta_{(0)}\sinh{R}$ is strictly increasing, $\eta_{(0)}\tanh{(R/2)}$ is constant, and $\eta_{(0)}\sinh{(R/2)}$ is strictly increasing.
\end{rem}

\subsection{Monotonicity of Steklov eigenvalues on geodesic disks with varying curvature}\label{subsec1.3} Compared with the Bandle-type bounds for Neumann eigenvalues \cite{Ban72} \cite[Cor.~3.9]{Ban80} and Cheng-type bounds for Dirichlet eigenvalues \cite{Che75} derived in \cite{LL23b}, the bounds for Steklov eigenvalues in this subsection may be called Escobar-type in view of \cite[Thms.~1 and 2]{Esc00}. For the results in this subsection, we need the metric expression of spaces of constant curvature $K\in \R$ as the warped product Riemannian manifold. Let $M^n_K=[0,r_0)\times \SS^{n-1}$ be the space form of constant curvature $K\in \R$. Then its warped product metric is given by
\begin{align*}
g=dr^2+h_K^2(r)g_{\SS^{n-1}},
\end{align*}
where
\begin{align*}
h_K(r)=
\begin{cases}
\dfrac{1}{\sqrt{K}}\sin (\sqrt{K}r),\quad &r\in [0,\dfrac{\pi}{\sqrt{K}}),\quad r_0=\dfrac{\pi}{\sqrt{K}},\quad \text{if }K>0,\\
 r,\quad &r\in [0,+\infty),\quad r_0=+\infty,\quad \text{if }K=0,\\
\dfrac{1}{\sqrt{-K}}\sinh (\sqrt{-K}r),\quad &r\in [0,+\infty),\quad r_0=+\infty,\quad \text{if }K<0.
\end{cases}
\end{align*}
We may prove the following result for the second-order Steklov eigenvalue.
\begin{corr}\label{thm7}
For the second-order Steklov eigenvalue $\sigma_{(m)}$ $(m \geq 1)$, we have the following results.
\begin{enumerate}
  \item Fix $A>0$. The $m$th Steklov eigenvalue $\sigma_{(m)}(K;A)$ on a geodesic disk of area $A$ in the space form $M_K^2$ is a strictly increasing  function of the curvature $K\in (-\infty,4\pi/A)$.
  \item Fix $\rho>0$. The $m$th Steklov eigenvalue $\sigma_{(m)}(K;\rho)$ on a geodesic disk of radius $\rho$ in the space form $M_K^2$ is a strictly increasing  function of the curvature $K\in (-\infty,(\pi/\rho)^2)$.
\end{enumerate}
\end{corr}
Next we consider the type one fourth-order Steklov eigenvalue problem.
\begin{corr}\label{thm8}
For the fourth-order Steklov eigenvalue $\xi_{(m)}$ $(m \geq 1)$, we have the following results.
\begin{enumerate}
  \item Fix $A>0$. The $m$th Steklov eigenvalue $\xi_{(m)}(K;A)$ on a geodesic disk of area $A$ in the space form $M_K^2$ is a strictly increasing function of the curvature $K\in (-\infty,4\pi/A)$.
   \item Fix $\rho>0$. The $m$th Steklov eigenvalue $\xi_{(m)}(K;\rho)$ on a geodesic disk of radius $\rho$ in the space form $M_K^2$ is a strictly increasing  function of the curvature $K\in (-\infty,(\pi/\rho)^2)$.
\end{enumerate}
\end{corr}
Last we consider the type two fourth-order Steklov eigenvalue problem.
\begin{corr}\label{thm9}
For the fourth-order Steklov eigenvalue $\eta_{(m)}$ $(m \geq 1)$, we have the following results.
\begin{enumerate}
   \item Fix $A>0$. Let the $m$th Steklov eigenvalue $\eta_{(m)}(K;A)$ on a geodesic disk of area $A$ in the space form $M_K^2$ be the function of the curvature $K\in (-\infty,4\pi/A)$. For $m=1$, there exists a $K_1>0$, such that $\eta_{(m)}(K;A)$ is strictly decreasing on $(-\infty, K_1)$, and strictly increasing on $(K_1,4\pi/A)$; for $m=2$, $\eta_{(m)}(K;A)$ is strictly decreasing on $(-\infty, 0)$, and strictly increasing on $(0,4\pi/A)$; while for $m\geq 3$, there exists a $K_m<0$, such that $\eta_{(m)}(K;A)$ is strictly decreasing on $(-\infty, K_m)$, and strictly increasing on $(K_m,4\pi/A)$.
  \item Fix $\rho>0$. Let the $m$th Steklov eigenvalue $\eta_{(m)}(K;\rho)$ on a geodesic disk of radius $\rho$ in the space form $M_K^2$ be the function of the curvature $K\in (-\infty,(\pi/\rho)^2)$. For $m=1$, $\eta_{(m)}(K;\rho)$ is strictly decreasing on $(-\infty,0)$, and strictly increasing on $(0,(\pi/\rho)^2)$; while for $m\geq 2$, there exists a $\bar{K}_m<0$, such that $\eta_{(m)}(K;\rho)$ is strictly decreasing on $(-\infty,\bar{K}_m)$, and strictly increasing on $(\bar{K}_m,(\pi/\rho)^2)$.
\end{enumerate}
\end{corr}
\begin{rem}
From Corollary ~\ref{thm9} we find that both $\eta_{(m)}(K;A)$ and $\eta_{(m)}(K;\rho)$ are unimodal with respect to $K$. Moreover, numerical experiments indicate that the location of the peak moves leftwards as $m$ increases.
\end{rem}
\begin{rem}
For $m=0$, recall
\begin{align*}
h(r) =
\begin{cases}
\dfrac{1}{\sqrt{-K}}\sinh (\sqrt{-K}r), & K < 0, \\[6pt]
\dfrac{1}{\sqrt{K}}\sin (\sqrt{K}r), & K > 0.
\end{cases}
\end{align*}
Substituting these into ~\eqref{eq-eta0} below (noting that $n=2$), we obtain
\begin{align*}
\eta_{(0)} =
\begin{cases}
\sqrt{-K}\coth\dfrac{\sqrt{-K}R}{2}, & K < 0, \\[6pt]
\sqrt{K}\cot\dfrac{\sqrt{K}R}{2}, & K > 0.
\end{cases}
\end{align*}
Then we may check that the following conclusions hold.
\begin{enumerate}
\item Fix $A>0$. The function $\eta_{(0)}(K;A)$ is strictly decreasing on $(-\infty,4\pi/A)$.
\item Fix $\rho>0$. The function $\eta_{(0)}(K;\rho)$ is strictly decreasing on $(-\infty,(\pi/\rho)^2)$.
\end{enumerate}
\end{rem}

\subsection{Scaling inequalities for $n\geq 3$}\label{subsec1.4}
We first consider second-order Steklov eigenvalues, and we have the following conclusion.

\begin{thm}\label{thm 10}
For the case $n\geq 3$ and $m\geq 1$, the function $\sigma_{(m)} \sin R$ is strictly increasing on $(0,\pi)$ in the sphere $\SS^n$ and the function $\sigma_{(m)} \sinh R$ is strictly decreasing on $(0,+\infty)$ in the hyperbolic space $\H^n$.
\end{thm}
\begin{rem}From Theorem ~\ref{thm 10} we have the following results.

\begin{enumerate}

\item  Since the functions $R/\sin R$, $\tan( R/2)/\sin R$, and $\sin( R/2)/ \sin R$ are strictly increasing on $(0,\pi)$, one can check that $\sigma_{(m)}R$, $\sigma_{(m)}\sin R$, $\sigma_{(m)}\tan(R/2)$, and $\sigma_{(m)} \sin(R/2)$ are strictly increasing on $(0,\pi)$.

\item Since the functions $R/\sinh R$, $\tanh (R/2)/\sinh R$, and $\sinh(R/2)/ \sinh R$ are strictly decreasing on $(0,+\infty)$, one can check that $\sigma_{(m)}R$, $\sigma_{(m)}\sinh R$, $\sigma_{(m)}\tanh(R/2)$, and $\sigma_{(m)}\sinh(R/2)$ are strictly decreasing on $(0,+\infty)$.

\end{enumerate}
\end{rem}

Moreover, for the higher-dimensional hyperbolic space, we are able to get one monotonicity result for $\eta_{(m)}$.
\begin{thm}\label{thm 11}
For the case $n\geq 3$ and $m\geq 0$, the function $\eta_{(m)} \sinh R$ is strictly increasing on $(0,+\infty)$ in the hyperbolic space $\H^n$.
\end{thm}

\subsection{Sharp bounds for eigenvalues on warped product manifolds}\label{subsec1.5}
Now let $M^n=[0,R]\times \SS^{n-1}$ be an $n$-dimensional $(n\geq 2)$ smooth Riemannian manifold equipped with the warped product metric
\begin{align*}
g=dr^2+h^2(r)g_{\SS^{n-1}},
\end{align*}
where $h(r)$ satisfies Assumption (A) given above.

In this part, first we get the following sharp bounds for the fourth-order Steklov eigenvalue $\xi_{(m)}$.
\begin{thm}\label{thm 14}
Let $M^n=[0,R]\times \SS^{n-1}$ be an $n$-dimensional ($n\geq 2$) smooth Riemannian manifold equipped with the warped product metric
\begin{equation*}
g=dr^2+h^2(r)g_{\SS^{n-1}},
\end{equation*}
where the warping function $h(r)$ satisfies Assumption (A). Suppose that $M$ has nonnegative Ricci curvature and a strictly convex boundary. Denote by $\xi_{(m)}$ the $m$th eigenvalue of the Steklov problem \eqref{problem2} without counting multiplicity. Then for $m\geq 1$ we have
\begin{equation}\label{ineq-xi2}
m^2(2+2m)\frac{h'(R)}{h^3(R)}\leq\xi_{(m)}\leq m^2(2+2m)\frac{1}{h^3(R)}, \quad n=2,
\end{equation}
and
\begin{align}\label{ineq-xi}
\xi_{(m)} \geq
\begin{cases}
m^2(3+2m)\dfrac{h'(R)}{h^3(R)}, & n = 3, \\[6pt]
m^2(n+2m)\dfrac{1}{h^3(R)}, & n \geq 4,
\end{cases}
\end{align}
with any equality holding only if $M$ is isometric to the Euclidean ball.
\end{thm}
\begin{rem}
The lower bound in \eqref{ineq-xi2} was obtained in \cite[Thm.~6]{Xio22} using a different argument, and the results in \eqref{ineq-xi} are stronger than those in \cite[Thm.~6]{Xio22}. Besides, geometrically $h'(R)/h(R)$ is the principal curvature of the boundary $\pt M$ in $M$, and $h(R)=(|\pt M|/|\SS^{n-1}|)^{1/(n-1)}$ can be viewed as the normalized boundary area.
\end{rem}
\begin{rem}
Assumptions as in Theorem ~\ref{thm 14} except $\mathrm{Ric}_g \geq 0$ replaced by $\mathrm{Ric}_g\leq 0$. Arguing as in the proof of Theorem ~\ref{thm 14}, we obtain that for $m\geq 1$,
\begin{equation}\label{ineq-xi2'}
m^2(2+2m)\frac{h'(R)}{h^3(R)}\geq\xi_{(m)}\geq m^2(2+2m)\frac{1}{h^3(R)}, \quad n=2,
\end{equation}
and
 \begin{align}\label{ineq-xi'}
\xi_{(m)} \leq
\begin{cases}
m^2(n+2m)\dfrac{h'(R)}{h^3(R)}, & n = 3, \\[6pt]
m^2(n+2m)\dfrac{1}{h^3(R)}, & n \geq 4.
\end{cases}
\end{align}
And any equality holds only if $M$ is isometric to the Euclidean ball. The details are left to the interested readers.
\end{rem}

For $n=3$ and $m=1$, the result in Theorem~\ref{thm 14} verifies the following Wang--Xia's conjecture in the case of warped product manifolds.

\begin{conj}[{Qiaoling~Wang and Changyu~Xia \cite[p.~13]{WX18}}]\label{conj2}
Let $(M^n,g)$ $(n\geq 2)$ be a connected compact smooth Riemannian manifold with boundary. Assume that $\mathrm{Ric}_g\geq 0$ and that the principal curvatures of the boundary $\pt M$ are bounded below by a constant $c>0$. Denote by $\lambda_1=\lambda_1(\pt M)$ the first nonzero eigenvalue of the Laplacian of $\pt M$. Then the first nonzero Steklov eigenvalue $\xi_1$ has a lower bound
\begin{equation*}
\xi_1\geq \frac{n+2}{n-1}c\lambda_1,
\end{equation*}
with the equality only for the Euclidean ball of radius $1/c$.
\end{conj}
Wang--Xia's conjecture can be viewed as the fourth-order counterpart of the famous Escobar's conjecture on the sharp lower bound for the first non-zero eigenvalue of the second-order Steklov eigenvalue problem \eqref{problem1} and is harder than Escobar's conjecture \cite[p.~115]{Esc99}; see \cite{XX24} for the recent progress on Escobar's conjecture. Regarding Wang--Xia's conjecture, previously Wang and Xia \cite{XW13} proved the non-sharp lower bound $\xi_1>nc\lambda_1/(n-1)$ in 2013 using the Reilly's formula. The second-named author \cite{Xio22} verified Wang--Xia's conjecture for warped product manifolds of the dimension $n=2$ or $n\geq 4$, leaving the case $n=3$ unsolved.

Second we deduce a two-sided inequality for the fourth-order Steklov eigenvalue $\eta_{(m)}$, and a sharp lower bound on the eigenvalue ratio for $\eta_{(m)}$.
\begin{thm}\label{thm 12}
Let $M^n=[0,R]\times \SS^{n-1}$ be an $n$-dimensional ($n\geq 2$) smooth Riemannian manifold equipped with the warped product metric
\begin{equation*}
g=dr^2+h^2(r)g_{\SS^{n-1}},
\end{equation*}
where the warping function $h(r)$ satisfies Assumption (A). Suppose that $M$ has nonnegative Ricci curvature and a strictly convex boundary. Denote by $\eta_{(m)}$ the $m$th eigenvalue of the Steklov problem \eqref{problem3} without counting multiplicity. Then for $m\geq 0$ we have
\begin{align}\label{ineq-eta}
(n+2m)\frac{h'(R)}{h(R)} \leq \eta_{(m)} \leq (n+2m)\frac{1}{h(R)},
\end{align}
and
\begin{align}\label{ineq-eta-ratio}
\frac{\eta_{(m+1)}}{\eta_{(m)}} \geq \frac{n+2m+2}{n+2m},
\end{align}
with the equality only if $M$ is isometric to the Euclidean ball.
\end{thm}
\begin{rem}
The lower bound of $\eta_{(0)}$ in \eqref{ineq-eta} was proved in \cite[Thm.~1.2]{WX09} for a general setting. For the upper bound of $\eta_{(0)}$, as noted by Kuttler \cite{Kut72} as well as Wang and Xia \cite[Thm.~1.3]{WX09}, the first eigenvalue satisfies an isoperimetric inequality $\eta_{(0)}\leq |\pt M|/|M|$ for a general Riemannian manifold. However, this upper bound of theirs is trivial in our setting of warped product manifolds, since we have the equality due to the expression of $\eta_{(0)}$ in \eqref{eq-eta0} below. Besides, for the case $n\geq 3$ and $m\geq 2$ the lower bound in \eqref{ineq-eta} improves on \cite[Thm.~7]{Xio22}, and for $n=2$ the bound \eqref{ineq-eta-ratio} was obtained in \cite[Thm.~8]{Xio21} using a different argument.
\end{rem}
\begin{rem}
Assumptions as in Theorem ~\ref{thm 12} except $\mathrm{Ric}_g \geq 0$ replaced by $\mathrm{Ric}_g\leq 0$. Arguing as in the proof of Theorem ~\ref{thm 12} (for $n=2$, the lower bound in \eqref{ineq-eta'} needs a special argument similar to that for the upper bound in ~\eqref{ineq-xi2}), we obtain that for $m\geq 1$,
\begin{align}\label{ineq-eta'}
(n+2m)\frac{1}{h(R)} \leq \eta_{(m)} \leq (n+2m)\frac{h'(R)}{h(R)}, \text{ for } n=2 \text{ or } n\geq 4,
\end{align}
\begin{align}\label{ineq-eta''}
\eta_{(m)} \leq (n+2m)\frac{h'(R)}{h(R)}, \text{ for } n=3.
\end{align}
For $m=0$, we have
\begin{align}\label{ineq-eta'''}
n\frac{1}{h(R)} \leq \eta_{(0)} \leq n\frac{h'(R)}{h(R)}, \text{ for all } n\geq 2.
\end{align}
And any equality holds only if $M$ is isometric to the Euclidean ball. Here we note that, when applied in the case $n=3$, the method used in the proof of Theorem ~\ref{thm 12} yields only a one-sided bound for $\eta_{(m)}$ $(m\geq 1)$. Moreover, this approach does not provide a ratio estimate for $\eta_{(m+1)}/\eta_{(m)}$ $(m\geq0)$ in any dimension $n$. The details are left to the  interested readers.
\end{rem}

The remaining part of the paper is organized as follows. In Section~\ref{sec2} we give some preliminaries including basic facts on warped product manifolds, geodesic balls and geometric factors in space forms, and the Steklov eigenvalue problems in consideration. In Section~\ref{sec3} we prove Theorems~\ref{thm2} and \ref{thm3}, and in Section~\ref{sec4} we prove Theorems~\ref{thm5} and \ref{thm6}. Next in Section~\ref{sec5} we give the proofs of Corollaries~\ref{thm7}, \ref{thm8} and \ref{thm9}. In Section~\ref{sec6} we discuss scaling inequalities in the higher-dimensional cases and prove Theorems~\ref{thm 10} and \ref{thm 11}. In Section~\ref{sec7} we present the proofs of sharp bounds on warped product manifolds, i.e., Theorems~\ref{thm 14} and \ref{thm 12}. In the last Section~\ref{sec8} we mention some open problems.
\begin{ack}
This research was supported by NSFC (Grant no. 12171334) and National Key R and D Program of China 2021YFA1001800.
\end{ack}

\section{Preliminaries}\label{sec2}
In this section we present some preliminaries which are needed in the proofs of our main results.
\subsection{The general setting of warped product manifolds}\label{subsec2.1} In this paper we are mainly concerned with warped product manifolds which include the sphere and the hyperbolic space as special cases. Let $M^n=[0,R]\times \SS^{n-1}$ be an $n$-dimensional $(n\geq 2)$ smooth Riemannian manifold equipped with the warped product metric
\begin{align*}
g=dr^2+h^2(r)g_{\SS^{n-1}},
\end{align*}
where the warping function $h(r)$ satisfies the assumption
\begin{itemize}
  \item[(A)] $h\in C^\infty([0,R])$, $h(r)>0$ for $r\in (0,R]$, $h'(0)=1$ and $h^{(2k)}(0)=0$ for all integers $k\geq 0$.
\end{itemize}
The Assumption (A) is imposed to guarantee that $M$ is a topological ball and is smooth at the origin; see Section~4.3.4 in Petersen's book \cite{Pet16}.

The following lemma, which is Lemma~8 in \cite{Xio22}, will be important in the proofs of Theorem ~\ref{thm 14} and Theorem~\ref{thm 12}.
\begin{lem}[{\cite[Lem.~8]{Xio22}}]\label{lem1}
Let $M^n=[0,R]\times \SS^{n-1}$ be an $n$-dimensional ($n\geq 2$) smooth Riemannian manifold equipped with the warped product metric
\begin{equation*}
g=dr^2+h^2(r)g_{\SS^{n-1}},
\end{equation*}
where the warping function $h(r)$ satisfies Assumption (A). Assume that $\mathrm{Ric}_g\geq 0$ and that the boundary $\pt M$ is strictly convex. Then we have
\begin{equation}
h''(r)\le 0,\quad 0<h'(r)\le 1,\quad r\in[0,R].
\end{equation}
\end{lem}

\subsection{Geodesic balls and geometric factors in the sphere and the hyperbolic space}\label{subsec2.2}
For the warped product manifold in Section~\ref{subsec2.1}, if $h(r)=\sin r$, we get the geodesic ball in the sphere
\begin{align*}
\SS^n=\{x\in \R^{n+1}|\sum_{i=1}^{n+1}(x^i)^2=1\}
\end{align*}
with the metric induced by the metric
\begin{align*}
ds^2=\sum_{i=1}^{n+1}(dx^i)^2
\end{align*}
on the Euclidean space $\R^{n+1}$. It is also called the spherical cap of the geodesic radius or the aperture $R$. The geometric factors mentioned before in the $2$-dimensional case include the geodesic radius $R$, the Euclidean radius of the boundary circle $\sin R$, the stereographic radius $\tan (R/2)$, and the area $4\pi \sin^2 (R/2)$ (see Figure~\ref{Fig1}).

If $h(r)=\sinh r$, we get the geodesic ball in the hyperbolic space
\begin{align*}
\H^n=\{x\in \R^{n+1}_1|x^{n+1}=\sqrt{1+\sum_{i=1}^{n}(x^i)^2}\}
\end{align*}
with the metric induced by the metric
\begin{align*}
ds^2=\sum_{i=1}^{n}(dx^i)^2-(dx^{n+1})^2
\end{align*}
on the Lorentzian space $\R^{n+1}_1$. The geometric factors for the $2$-dimensional geodesic disk include the geodesic radius $R$, the Euclidean radius of the boundary circle $\sinh R$, the stereographic radius $\tanh (R/2)$, and the area $4\pi \sinh^2 (R/2)$ (see Figure~\ref{Fig2}).

\begin{figure}
\centering
\begin{tikzpicture}[scale=3.2]

 \draw [<-] (0,1.55) node [right] {$z$}--(0,-1.25) ;
 \draw [->] (-1.2,0)--(1.2,0)node [right] {$x$};
 \node [below left] at (0,0) {$0$};
 \node [below left] at (0,1) {$1$};
 \node [below right] at (0,1) {$N$};
 \node [below right] at (0,-1) {$S$};
 \node [below right] at (0.25,1.03) {$R$};
 \node [below left] at (-0.09,1) {$\mathrm{Area}=4\pi \sinh^2(R/2)$};
 \draw [<-] (0.17,-0.05)-- (0.35,-0.2) node [below right] {$\tanh (R/2)$};
 \draw [<-] (0.3,1.28)--(0.4,1.4) node [above right] {$\sinh R$};
 \draw [dashed] (0,0)--(0.75,1.25);
 \draw [dashed] (0,-1)--(0.75,1.25);
 \draw [dashed] (0,1.25)--(0.75,1.25);
 \draw [-] (-0.75,1.25) to [out=-20, in=200] (0.75,1.25);
 \draw [dashed] (-0.75,1.25) to [out=15, in=165] (0.75,1.25);
 \draw [domain=-0.9:0.9] plot (\x,{sqrt(1+\x*\x)});
\end{tikzpicture} 
\caption{Geometric factors in $\H^2$}\label{Fig2}
\end{figure}
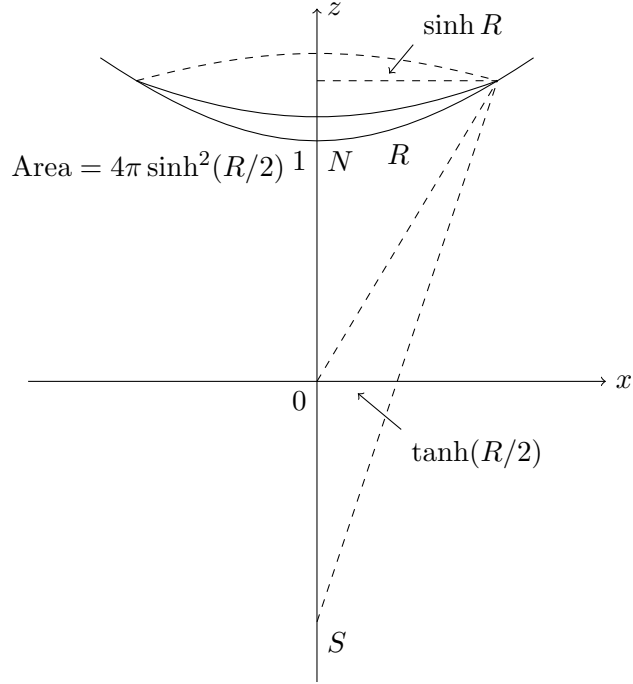

%

\subsection{Second-order and fourth-order Steklov eigenvalue problems}\label{subsec2.3}
 Let $(M^n,g)$ be an $n$-dimensional Riemannian manifold with boundary $\pt M$. We consider the second-order Steklov eigenvalue problem
\begin{equation}\label{problem2.2}
	\begin{cases}
		\Delta u =0,&\text{ in }  M, \\
		\dfrac{\partial u }{\partial \nu }=\sigma u, &\text{ on }  \partial M,
	\end{cases}
\end{equation}
 where $\nu$ denotes the outward unit normal to $\partial M$.

 The spectrum of the classical second-order Steklov eigenvalue problem is nonnegative, discrete and unbounded (counting multiplicity):
 \begin{align}
    0=\sigma_0<\sigma_1\le\sigma_2\le\dots\nearrow+\infty.\nonumber
\end{align}
We denote by $\sigma_{(m)}$ the eigenvalues without counting multiplicity. For instance, for the $n$-dimensional Euclidean ball $B_R$ with radius $R$, we have
\begin{equation*}
\sigma_{(0)}=\sigma_0=0,\quad \sigma_{(1)}=\sigma_1=\cdots=\sigma_{n}=\frac{1}{R},
\end{equation*}
and $\sigma_{(m)}=m/R$ with multiplicity $C_{n+m-1}^{n-1}-C_{n+m-3}^{n-1}$ for $m\geq 2$ (see e.g. \cite{GP17}).

For warped product manifolds, we have the following characterization of all its eigenfunctions and eigenvalues.
\begin{prop}[{\cite[Lem.~3]{Esc00}\cite[Prop.~9]{Xio22}}]\label{prop1}
Let $M^n=[0,R]\times \SS^{n-1}$ be an $n$-dimensional ($n\geq 2$) smooth Riemannian manifold equipped with the warped product metric
\begin{equation*}
g=dr^2+h^2(r)g_{\SS^{n-1}},
\end{equation*}
where the warping function $h(r)$ satisfies Assumption (A). Then any nontrivial eigenfunction $\varphi$ of the problem \eqref{problem2.2} can be written as $\varphi(r,p)=\psi(r)\omega(p)$, where $\omega$ is a spherical harmonic on $\SS^{n-1}$ of some degree $m\geq 1$, i.e.,
\begin{equation*}
-\Delta_{\SS^{n-1}}\omega=\tau_m\omega \text{ on }\SS^{n-1},\quad \tau_m=m(n-2+m),
\end{equation*}
and $\psi$ is a nontrivial solution of the ODE
\begin{equation}\label{2nd}
\begin{cases}
\dfrac{1}{h^{n-1}}\dfrac{d}{dr}(h^{n-1}\dfrac{d}{dr}\psi)-\dfrac{\tau_m\psi}{h^2}=0,\quad &r\in (0,R],\\
\psi(0)=0.&
\end{cases}
\end{equation}
For any nontrivial solution $\psi$ to the above ODE, the $m$th eigenvalue $\sigma_{(m)}$ without counting multiplicity is given by $\sigma_{(m)}=\psi'(R)/\psi(R)$.
\end{prop}

Next we consider the type one fourth-order Steklov eigenvalue problem
\begin{equation}\label{type1}
    \begin{cases}
        \Delta^{2} u =0, \text{ in } M,\\
      \dfrac{\partial u}{\partial\nu}=0,\quad  \dfrac{\partial(\Delta u)}{\partial \nu} +\xi u=0, \text{ on } \partial M.
    \end{cases}
\end{equation}
It was pioneered by Kuttler and Sigillito in 1968 \cite{KS68}, and its eigenvalues form a discrete, increasing sequence (counting multiplicity):
\begin{align}
    0=\xi_0<\xi_1\le\xi_2\le\dots\nearrow+\infty.\nonumber
\end{align}
We also use $\xi_{(m)}$ to denote the eigenvalues without counting multiplicity. For the $n$-dimensional Euclidean ball $B_R$ with radius $R$, we have $\xi_{(m)}=m^2(n+2m)/R^3$ with multiplicity $C_{n+m-1}^{n-1}-C_{n+m-3}^{n-1}$ (see \cite[Theorem~1.5]{XW18}).

We have the following characterization of all its eigenfunctions and eigenvalues.
\begin{prop}[{\cite[Prop.~11]{Xio22}}]\label{prop3}
Let $M^n=[0,R]\times \SS^{n-1}$ be an $n$-dimensional ($n\geq 2$) smooth Riemannian manifold equipped with the warped product metric
\begin{equation*}
g=dr^2+h^2(r)g_{\SS^{n-1}},
\end{equation*}
where the warping function $h(r)$ satisfies Assumption (A). Suppose that $M$ has nonnegative Ricci curvature and a strictly convex boundary. Then any nontrivial eigenfunction $\varphi$ of the problem \eqref{type1} can be written as $\varphi(r,p)=\psi(r)\omega(p)$, where $\omega$ is a spherical harmonic on $\SS^{n-1}$ of some degree $m\geq 1$, i.e.,
\begin{equation*}
-\Delta_{\SS^{n-1}}\omega=\tau_m\omega \text{ on }\SS^{n-1},\quad \tau_m=m(n-2+m),
\end{equation*}
and $\psi$ is a nontrivial solution of the ODE
\begin{equation}\label{eq4}
\begin{cases}
\dfrac{1}{h^{n-1}}\dfrac{d}{dr}(h^{n-1}\dfrac{d}{dr}\psi)-\dfrac{\tau_{m}\psi}{h^2}=\tilde{\psi},&\\
\dfrac{1}{h^{n-1}}\dfrac{d}{dr}(h^{n-1}\dfrac{d}{dr}\tilde{\psi})-\dfrac{\tau_m\tilde\psi}{h^2}=0,&\\
\psi(0)=0,\quad  \psi'(R)=0,\quad \tilde{\psi}(0)=0.&
\end{cases}
\end{equation}
For any nontrivial solution $\psi$ to the above ODE, the $m$th eigenvalue $\xi_{(m)}$ without counting multiplicity is given by $\xi_{(m)}=-\tilde{\psi}'(R)/\psi(R)$.
\end{prop}

Then we consider the type two fourth-order Steklov eigenvalue problem
\begin{equation}\label{type2}
    \begin{cases}
        \Delta^{2} u =0, \text{ in }M,\\
 u=0,\quad \Delta u =\eta \dfrac{\partial u}{\partial \nu}, \text{ on }\partial M.
    \end{cases}
\end{equation}
It was initially studied by Kuttler and Sigillito in 1968 \cite{KS68} and by Payne in 1970 \cite{Pay70}, and its eigenvalues form a discrete, increasing sequence (counting multiplicity):
\begin{align}
    0<\eta_0<\eta_1\le\eta_2\le\dots\nearrow+\infty.\nonumber
\end{align}
Note that the first eigenvalue $\eta_0$ is positive and simple (see \cite[Theorem~1]{BGM06} or \cite{RS15}). We also use $\eta_{(m)}$ to denote the eigenvalues without counting multiplicity. For the $n$-dimensional Euclidean ball $B_R$ with radius $R$, we know $\eta_{(m)}=(n+2m)/R$ with multiplicity $C_{n+m-1}^{n-1}-C_{n+m-3}^{n-1}$ (see \cite[Theorem~1.3]{FGW05}).

Similarly, we have the following characterization of all its eigenfunctions and eigenvalues.
\begin{prop}[{\cite[Prop.~14]{Xio22}}]\label{prop4}
Let $M^n=[0,R]\times \SS^{n-1}$ be an $n$-dimensional ($n\geq 2$) smooth Riemannian manifold equipped with the warped product metric
\begin{equation*}
g=dr^2+h^2(r)g_{\SS^{n-1}},
\end{equation*}
where the warping function $h(r)$ satisfies Assumption (A). Suppose that $M$ has nonnegative Ricci curvature and a strictly convex boundary. Then the first eigenfunction $\varphi_0(r,p)$ of the problem \eqref{type2} is given by $\psi_0(r)$ up to a constant multiple, where $\psi_0$ is a nontrivial solution of the ODE
\begin{equation}\label{eq5.3}
\begin{cases}
\dfrac{1}{h^{n-1}}\dfrac{d}{dr}(h^{n-1}\dfrac{d}{dr}\psi)=\tilde{\psi},&\\
\dfrac{1}{h^{n-1}}\dfrac{d}{dr}(h^{n-1}\dfrac{d}{dr}\tilde{\psi})=0,&\\
\psi'(0)=0,\quad \psi(R)=0,\quad \tilde{\psi}'(0)=0.&
\end{cases}
\end{equation}
Any higher-order eigenfunction $\varphi$ of the problem \eqref{type2} can be written as $\varphi(r,p)=\psi(r)\omega(p)$, where $\omega$ is a spherical harmonic on $\SS^{n-1}$ of some degree $m\geq 1$, i.e.,
\begin{equation*}
-\Delta_{\SS^{n-1}}\omega=\tau_m\omega \text{ on }\SS^{n-1},\quad \tau_m=m(n-2+m),
\end{equation*}
and $\psi$ is a nontrivial solution of the ODE
\begin{equation}\label{eq5.4}
\begin{cases}
\dfrac{1}{h^{n-1}}\dfrac{d}{dr}(h^{n-1}\dfrac{d}{dr}\psi)-\dfrac{\tau_{m}\psi}{h^2}=\tilde{\psi},&\\
\dfrac{1}{h^{n-1}}\dfrac{d}{dr}(h^{n-1}\dfrac{d}{dr}\tilde{\psi})-\dfrac{\tau_m\tilde\psi}{h^2}=0,&\\
\psi(0)=0,\quad  \psi(R)=0,\quad \tilde{\psi}(0)=0.&
\end{cases}
\end{equation}
For any nontrivial solution $\psi$ to the ODE \eqref{eq5.3} or \eqref{eq5.4}, the $m$th eigenvalue $\eta_{(m)}$ ($m\geq 0$) without counting multiplicity is given by $\eta_{(m)}=\tilde{\psi}(R)/\psi'(R)$.
\end{prop}


To give proofs for  Theorem~\ref{thm 11}, Theorem ~\ref{thm 14} and Theorem~\ref{thm 12}, we need the following key lemma.
\begin{lem}\label{lem4}
Let $M^n=[0,R]\times \SS^{n-1}$ be an $n$-dimensional ($n\geq 2$) smooth Riemannian manifold equipped with the warped product metric
\begin{equation*}
g=dr^2+h^2(r)g_{\SS^{n-1}},
\end{equation*}
where the warping function $h(r)$ satisfies Assumption (A). Suppose that $M$ has nonnegative Ricci curvature and a strictly convex boundary. For the fourth-order Steklov eigenvalues $\xi_{(m)}$ and $\eta_{(m)}$, we have
\begin{align}
\xi_{(m)}&=\frac{h^{n-1}(R)(u_m'(R))^2}{\int_0^R h^{n-1}u_m^2 dr},\quad m\geq 1,\label{eq-xi}\\
\eta_{(m)}&=\frac{h^{n-1}(R)u_m^2(R)}{\int_0^R h^{n-1}u_m^2 dr},\quad m\geq 1,\label{eq-eta}
\end{align}
where $u_m$ is a non-trivial solution of
\begin{equation}\label{eq-ode}
\left\{
\begin{aligned}
&u_m'' + \frac{(n-1)h'}{h}\,u_m' - \frac{\tau_m}{h^2}u_m = 0, \\[6pt]
&u_m(0)=0,
\end{aligned}
\right.
\end{equation}
and
\begin{equation}\label{eq-eta0}
\eta_{(0)}=\frac{h^{n-1}(R)}{\int_0^R h^{n-1} dr}.
\end{equation}
\end{lem}
\begin{rem}
The assumption that $\mathrm{Ric}_g\geq 0$ and $\pt M$ is strictly convex is indeed unnecessary. This assumption is imposed in order to invoke Propositions~\ref{prop3} and \ref{prop4} which corresponds to respectively Propositions~11 and 14 in \cite{Xio22}. By careful inspecting the proofs of Propositions~11 and 14 in \cite{Xio22}, we may see that this assumption is only used to determine the order of the eigenvalues. Here in the proof below, we may first derive the formulas \eqref{eq-xi} and \eqref{eq-eta}, and then use ODE comparison arguments directly to determine the order of the eigenvalues. The details are left to the interested readers.
\end{rem}
\begin{rem}
In some of the proofs below, we will assume without loss of generality and tacitly that $u_m(r)>0$ and $u_m'(r)>0$ for $r\in (0,R]$.
\end{rem}
\begin{proof}[Proof of Lemma ~\ref{lem4}]
%
%
%

First for the eigenvalue $\xi_{(m)}$ $(m\geq1)$, considering the ODE \eqref{eq4} and letting $\tilde{\psi}=u_m$ there, we know that
\begin{equation}\label{xixi}
\xi_{(m)}=-\frac{u_m'(R)}{\psi(R)}.
\end{equation}

Set
\begin{align*}
Ly=\frac{1}{h^{n-1}}\frac{d}{dr}\left(h^{n-1}\frac{dy}{dr}\right)
 - \frac{\tau_m y}{h^2}.
\end{align*}
For any two functions $p, q \in C^2([0,R])$, by direct computation we get
\begin{align*}
pLq-qLp=\frac{1}{h^{n-1}}\frac{d}{dr}\left(h^{n-1}\left(pq'-p'q\right)\right).
\end{align*}
Now choose
\begin{align*}
p=u_m,\quad  q=\psi.
\end{align*}
Then we have
\begin{equation}\label{eq6.1}
h^{n-1}u_m^2=\frac{d}{dr}\left(h^{n-1}\left(u_m\psi'-u_m'\psi\right)\right).
\end{equation}\label{2.12}
Integrating both sides in \eqref{eq6.1} over the interval $(0,R)$ and using the boundary condition
\begin{align*}
h(0)=0,\quad \psi(0)=0,\quad \psi'(R)=0,\quad u_m(0)=0,
\end{align*}
we have
\begin{equation}\label{2.14}
\int_0^R h^{n-1}u_m^2 dr=-h^{n-1}(R)u_m'(R)\psi(R).
\end{equation}
Combining ~\eqref{xixi} and ~\eqref{2.14} to eliminate $\psi(R)$, we obtain
\begin{align*}
\xi_{(m)}&=\frac{h^{n-1}(R)(u_m'(R))^2}{\int_0^R h^{n-1}u_m^2 dr}.
\end{align*}


Next for the eigenvalue $\eta_{(m)}$ ($m\geq 1$), considering the ODE \eqref{eq5.4} and letting $\tilde{\psi}=u_m$ there, we know that
\begin{equation}\label{2.15}
\eta_{(m)}=\frac{u_m(R)}{\psi'(R)}.
\end{equation}
Again integrating both sides in \eqref{eq6.1} with respect to $r$ from $0$ to $R$, we have
\begin{equation}\label{2.16}
\int_0^R h^{n-1}u_m^2 dr=h^{n-1}(R)u_m(R)\psi'(R).
\end{equation}
Combining ~\eqref{2.15} and ~\eqref{2.16} to eliminate $\psi'(R)$, we obtain
\begin{align*}
\eta_{(m)}&=\frac{h^{n-1}(R)u_m^2(R)}{\int_0^R h^{n-1}u_m^2 dr}.
\end{align*}

Last, when $m=0$, we may check that a similar argument as above works, yielding
\begin{align*}
\eta_{(0)}=\frac{h^{n-1}(R)}{\int_0^R h^{n-1} dr}.
\end{align*}
Thus we complete the proof.
\end{proof}


\section{Proofs of results in the $2$-dimensional sphere}\label{sec3}
For the proofs in this section, we need the following lemma.
\begin{lem}\label{lem2}
In the $2$-dimensional sphere, we have
\begin{equation}\label{eq3.4}
\xi_{(m)}=\frac{m^2(\tan\frac{R}{2})^{2m}}{\sin R\int_0^R(\tan\frac{r}{2})^{2m}\sin r dr},\quad m\geq 1,
\end{equation}
\begin{equation}\label{eq3.5}
\eta_{(m)}=\frac{(\tan\frac{R}{2})^{2m}\sin R}{\int_0^R(\tan\frac{r}{2})^{2m}\sin r dr},\quad m\geq 1,
\end{equation}
and
\begin{equation}\label{eq-eta_0}
\eta_{(0)}=\cot\frac{R}{2}.
\end{equation}
\end{lem}
\begin{proof}[Proof of Lemma~\ref{lem2}]
In the sphere $\SS^2$,  for $m\geq 1$ the equation \eqref{eq-ode} becomes
\begin{align*}
u_m''+\frac{u_m'}{\tan r}-\frac{m^2u_m}{\sin^2 r}=0.
\end{align*}
Introducing the change of variable $t=\tan(r/2)$, we have
\begin{align*}
(u_m)_{tt}+\frac{1}{t}(u_m)_t-\frac{m^2}{t^2}u_m=0.
\end{align*}
This is an Euler--Cauchy equation, whose general solution is explicitly given by
\begin{align*}
u_m(t)=C_1t^m+C_2t^{-m}.
\end{align*}
Imposing regularity at $t=0$ rules out the term $t^{-m}$, so we obtain
\begin{align*}
u_m(r)=C(\tan\frac{r}{2})^{m}.
\end{align*}
Substituting this expression into ~\eqref{eq-xi} and ~\eqref{eq-eta}, we complete the proof for $m\geq 1$. For $m=0$, taking $n=2$ and $h(r)=\sin r$ in ~\eqref{eq-eta0}, we may see that \eqref{eq-eta_0} holds. Thus we finish the proof.
\end{proof}
Before we proceed, for readers' convenience it is worth pointing out that the following computation formula will be used frequently in the proofs of Theorem ~\ref{thm2} and Theorem ~\ref{thm3}:
\begin{align}\label{eq3.1}
\frac{d((\tan\frac{R}{2})^{k})}{dR}=k\frac{(\tan\frac{R}{2})^{k-1}}{2\cos^2\frac{R}{2}}=k\frac{(\tan\frac{R}{2})^{k}}{\sin R}.
\end{align}


Now we are ready to present the proofs. Since the geometric factors involved in Theorem~\ref{thm3} have a lower degree, we choose to prove it first.
\begin{proof}[Proof of Theorem~\ref{thm3}]
(1) Let
\begin{equation}\label{eq3.7}
M(R)=\int_0^R(\tan\frac{r}{2})^{2m}\cdot\sin r dr.
\end{equation}
Then we consider the function
\begin{align*}
F(R)&=\eta_{(m)}R=\frac{R\cdot(\tan\dfrac{R}{2})^{2m}\cdot\sin R}{M(R)}.
\end{align*}
 We have
\begin{align*}
(\tan\frac{R}{2})^{-2m}F'(R)=\frac{(2m+\cos R)M(R)-(\tan\dfrac{R}{2})^{2m}\cdot\sin^2 R}{M^2(R)}\cdot R+\frac{\sin R}{M(R)}.
\end{align*}
Next, set
\begin{align*}
J(R)=(2m+\cos R)M(R)-(\tan\frac{R}{2})^{2m}\sin^2R.
\end{align*}
Then we have, by integration by parts,
\begin{align*}
J'(R)&=-\sin R\cdot(M(R)+(\tan\frac{R}{2})^{2m}\cos R)\\
&=-m\sin R\int_0^R \frac{(\tan\frac{r}{2})^{2m-1}\cos r}{\cos^2\frac{r}{2}}dr\\
&=-2m\sin R\int_0^{\tan\frac{R}{2}} \frac{t^{2m-1}(1-t^2)}{1+t^2}dt.
\end{align*}
Therefore, there exists an $R_m^{\circ}>0$, such that $J(R)$ is decreasing on $(0,R_m^{\circ})$ and increasing on $(R_m^{\circ},\pi)$. Next by direct verification we get
\begin{align}
M(R)&=\frac{1}{2(m-1)}(\tan \frac{R}{2})^{2m}\sin^2 R(1+o(1)), \text{ as }R \to \pi-,\quad m\geq 2,\label{M-m2}\\
M(R)&=-4\log \cos \frac{R}{2}+\cos R-1, \quad m=1.\label{M-m1}
\end{align}
Then we may check directly that $J(R)>0$ for $R$ close to $\pi$ and $m\geq 1$. Therefore, there exists an $R_m^{\ast}>0$, such that $J(R)<0$ on $(0,R_m^{\ast})$ and $J(R)>0$ on $(R_m^{\ast},\pi)$. It follows that $F'(R) > 0$ on $(R_m^{\ast},\pi)$. It remains only to prove that $F'(R)>0$ when $R<R_m^{\ast}$.

Next we consider the function
\begin{align*}
F_1(R)&=\frac{F'(R)(M(R))^2}{J(R)(\tan{\dfrac{R}{2}})^{2m}}=R+\frac{M(R)\sin R}{J(R)}.
\end{align*}
Then $F'(R)$ and $F_1(R)$ have the opposite signs on $(0,R_m^{\ast})$. Note
\begin{align*}
M(R)&=\frac{1}{2^{2m+1}(m+1)}R^{2m+2}(1+o(1)), \text{ as }R\to 0,\\
J(R)&=-\frac{1}{2^{2m+1}(m+1)}R^{2m+2}(1+o(1)), \text{ as }R\to 0.
\end{align*}
It follows that
\begin{align*}
F_1(R)=o(R), \text{ as }R\to 0.
\end{align*}
Then $F_1(0)=0$. Moreover, we have
\begin{align*}
    F_1'(R)=1+\frac{\left((\tan{\dfrac{R}{2}})^{2m}\sin^2R+M(R)\cos R\right)J(R)-J'(R)M(R)\sin R}{J^2(R)}.
\end{align*}
Using
\begin{align*}
(\tan\frac{R}{2})^{2m}\sin^2R=(2m+\cos R)M(R)-J(R)
\end{align*}
in the above equality, we get
\begin{align*}
F'_1(R)=\frac{M(R)}{J^2(R)}(2(m+\cos R)J-J'\sin R).
\end{align*}
Set
\begin{align*}
    F_2(R)=\frac{J^2(R)F'_1(R)}{M(R)}=2(m+\cos R)J-J'\sin R.
\end{align*}
By direct computation, we obtain
\begin{align*}
    F_2(R)&=(\cos^2 R+6m\cos R+4m^2+1)M(R)\\
    &\quad -(2m+\cos R)(\tan{\frac{R}{2}})^{2m}\sin^2R.
\end{align*}
Then, it follows that
\begin{align*}
    F_2'(R)=2\sin R\left((\tan{\dfrac{R}{2}})^{2m}\sin^2 R-(3m+\cos R)M(R)\right).
\end{align*}
Next, we define
\begin{align*}
    F_3(R)=\frac{(\tan{\dfrac{R}{2}})^{2m}\sin^2 R}{3m+\cos R}-M(R).
\end{align*}
Then we have
\begin{align*}
    F_3'(R)=\frac{(\tan{\dfrac{R}{2}})^{2m}\sin R}{(3m+\cos R)^2}(2m\cos R+1-3m^2)<0.
\end{align*}
In view of
\begin{align*}
    F_1(0)=0,\quad F_2(0)=0,\quad F_3(0)=0,
\end{align*}
we know $F_1(R)<0$ when $R<R_m^{\ast}$. It follows that $F'(R) > 0$ when $R<R_m^{\ast}$. So we complete the proof.

(2) Define
\begin{align*}
G(R)&=\eta_{(m)}\sin R=\frac{(\tan\dfrac{R}{2})^{2m}\sin^2R}{M(R)}.
\end{align*}
We have
\begin{align*}
G'(R)=(\tan\frac{R}{2})^{2m}\sin R\cdot\frac{(2m+2\cos R)M(R)-(\tan\dfrac{R}{2})^{2m}\sin^2R}{M^2(R)}.
\end{align*}
Set
\begin{align*}
G_1(R)=(2m+2\cos R)M(R)-(\tan\frac{R}{2})^{2m}\sin^2R.
\end{align*}
Then $G_1(0)=0$. In addition, we have
\begin{align*}
G_1'(R)&=-2\sin R\cdot M(R)< 0.
\end{align*}
Therefore, we have $G_1(R)<0$. Then $G'(R)< 0$. It follows that $G(R)$ is strictly decreasing on $(0,\pi)$.

(3) We may check that the function
\begin{align*}
R \mapsto \frac{\tan\frac{R}{2}}{R}
\end{align*}
increases strictly on $(0,\pi)$.
Then note that
\begin{align*}
\eta_{(m)}\tan\frac{R}{2}=\eta_{(m)}R \cdot \frac{\tan\frac{R}{2}}{R}.
\end{align*}
From (1) we complete the proof.

(4) Let
\begin{align*}
Q(R)&=\eta_{(m)}\sin\dfrac{R}{2}=\frac{(\tan\dfrac{R}{2})^{2m}\sin\dfrac{R}{2}\sin R}{M(R)}.
\end{align*}
By direct computation, we have
\begin{align*}
Q'(R)=&\frac{(\tan\dfrac{R}{2})^{2m}\sin\dfrac{R}{2}}{M^2(R)}\\
&\times  \left((2m+\cos R+\cos^2\frac{R}{2})M(R)-(\tan\frac{R}{2})^{2m}\sin^2R\right).
\end{align*}
Set
\begin{align*}
Q_1(R)=M(R)-\frac{2(\tan\frac{R}{2})^{2m}\sin^2R}{K(R)},
\end{align*}
where
\begin{align*}
K(R)=4m+1+3\cos R.
\end{align*}
Notice that $Q_1(0)=0$. In addition, we have
\begin{align*}
Q_1'(R)&=\frac{(\tan\frac{R}{2})^{2m}\sin R}{K^2(R)}(1-\cos R)(4m-5-3\cos R).
\end{align*}

For $m=1$, there exists a unique $R_1^\star>0$, such that $Q_1'(R_1^\star)=0$. Then $Q_1(R)$ is strictly decreasing on $(0,R_1^\star)$ and strictly increasing on $(R_1^\star,\pi)$. From \eqref{M-m1} we know $Q_1(\pi)=+\infty$. Then there exists an $R_1>0$, such that $\eta_{(1)}\sin(R/2)$ decreases strictly on $(0,R_1)$, and increases strictly on $(R_1,\pi)$.

For $m\geq 2$, we have $Q_1'(R)> 0$. Then $Q_1(R)> 0$. It follows that $\eta_{(m)}\sin(R/2)$ increases strictly on $(0,\pi)$.
\end{proof}

Next we give the proof of Theorem~\ref{thm2}.
\begin{proof}[Proof of  Theorem~\ref{thm2}]
(1) We may check that the function
$$R \mapsto \frac{R}{\sin R}$$
is strictly increasing on $(0,\pi)$. On the other hand, we have
\begin{align*}
\xi_{(m)} R^3=m^2\cdot\eta_{(m)} R\cdot\frac{R^2}{\sin^2 R}.
\end{align*}
Then $\xi_{(m)} R^3$ increases strictly on $(0,\pi)$.

(2) Note that
\begin{align*}
\xi_{(m)}\sin^3 R=m^2\eta_{(m)}\sin R.
\end{align*}
Then (2) holds because we have proved in Theorem ~\ref{thm3} (2) that $\eta_{(m)}\sin R$ is strictly decreasing on $(0,\pi)$.

(3) It is not hard to prove that the function
\begin{align*}
R\mapsto\frac{\tan\dfrac{R}{2}}{\sin R}
\end{align*}
is strictly increasing on $(0,\pi)$. Next note that
\begin{align*}
\xi_{(m)}(\tan\frac{R}{2})^3=m^2\eta_{(m)}\tan\frac{R}{2}\cdot\frac{(\tan\dfrac{R}{2})^2}{\sin^2R}.
\end{align*}
We can finish the proof of (3), since we have shown in Theorem \ref{thm3} (3) that $\eta_{(m)}\tan({R}/{2})$ is strictly increasing on $(0,\pi)$.

(4) Set
\begin{align*}
    q(R)&=\xi_{(m)}\sin^3\frac{R}{2}=\frac{m^2(\tan{\dfrac{R}{2}})^{2m+1}\sin{\dfrac{R}{2}}}{2M(R)}.
\end{align*}
Then we have
\begin{align*}
    q'(R)=\frac{m^2(\tan{\dfrac{R}{2}})^{2m+1}}{4\cos{\dfrac{R}{2}}M(R)^2}q_1(R),
\end{align*}
where
\begin{align*}
    q_1(R)=(2m+1+\cos^2{\frac{R}{2}})M(R)-(\tan{\dfrac{R}{2}})^{2m}\sin^2 R.
\end{align*}
Note that $q_1(0)=0$. In addition, we have
\begin{align*}
    q'_1(R)=-\frac{\sin R}{2}\left(M(R)+3(\tan{\dfrac{R}{2}})^{2m}(\cos R-1)\right).
\end{align*}
Set
\begin{align*}
    q_2(R)=M(R)+3(\tan{\dfrac{R}{2}})^{2m}(\cos R-1).
\end{align*}
Then $q_2(0)=0$. Moreover, we have
\begin{align*}
    q'_2(R)&=\frac{(\tan{\dfrac{R}{2}})^{2m}}{\sin R}(2\cos^2 R+6m\cos R-6m-2)<0.
\end{align*}
Therefore, we get $q_2(R) < 0$. Then $q_1(R) > 0$. It follows that $q'(R)>0$, so $q(R)$ increases strictly on $(0,\pi)$.

\end{proof}

\section{Proofs of results in the $2$-dimensional hyperbolic space}\label{sec4}

In the $2$-dimensional hyperbolic space, the following lemma will be required.
\begin{lem}\label{lem3}
In the $2$-dimensional hyperbolic space, we have
\begin{equation}\label{eq4.4}
\xi_{(m)}=\frac{m^2(\tanh\frac{R}{2})^{2m}}{\sinh R\int_0^R(\tanh\frac{r}{2})^{2m}\sinh r dr}, \quad m\geq 1,
\end{equation}
\begin{equation}\label{eq4.5}
\eta_{(m)}=\frac{(\tanh\frac{R}{2})^{2m}\sinh R}{\int_0^R(\tanh\frac{r}{2})^{2m}\sinh r dr}, \quad m\geq 1,
\end{equation}
and
\begin{equation}
\eta_{(0)}=\coth\frac{R}{2}.
\end{equation}
\end{lem}
\begin{proof}[Proof of Lemma~\ref{lem3}]
In the hyperbolic space $\H^2$, for $m\geq 1$ the equation \eqref{eq-ode} becomes
\begin{align*}
u_m''+\frac{u_m'}{\tanh r}-\frac{m^2u_m}{\sinh^2 r}=0.
\end{align*}
Consider the substitution $t=\tanh(r/2)$. Then arguing as in Lemma~\ref{lem2}, we get the desired formulas.
\end{proof}
As in Section~\ref{sec3}, in this section the following computation formula will be used frequently:
\begin{align*}\label{eq4.1}
\frac{d((\tanh\frac{R}{2})^{k})}{dR}=k\frac{(\tanh\frac{R}{2})^{k-1}}{2\cosh^2\frac{R}{2}}=k\frac{(\tanh\frac{R}{2})^{k}}{\sinh R}.
\end{align*}

Now we are in a position to give the proofs. Since the geometric factors involved in Theorem~\ref{thm6} have a lower degree, we choose to prove Theorem~\ref{thm6} first.
\begin{proof}[Proof of Theorem~\ref{thm6}]
(1) Let
\begin{equation}\label{eq 4.6}
M(R)=\int_0^R(\tanh\frac{r}{2})^{2m}\cdot\sinh r dr.
\end{equation}
Considering the function
\begin{align*}
F(R)&=\eta_{(m)}R=\frac{R\cdot(\tanh\frac{R}{2})^{2m}\cdot\sinh R}{M(R)},
\end{align*}
we have
\begin{align*}
(\tanh\frac{R}{2})^{-2m}F'(R)=\frac{(2m+\cosh R)M(R)-(\tanh\frac{R}{2})^{2m}\cdot\sinh^2 R}{M^2(R)}\cdot R+\frac{\sinh R}{M(R)}.
\end{align*}

Next, set
\begin{align*}
J(R)=(2m+\cosh R)M(R)-(\tanh\frac{R}{2})^{2m}\sinh^2R.
\end{align*}
Then we have, by integration by parts,
\begin{align*}
J'(R)&=\sinh R\cdot\left(M(R)-(\tanh\frac{R}{2})^{2m}\cosh R\right)\\
&=-2m\sinh R\int_0^R \frac{(\tanh\frac{r}{2})^{2m}}{\tanh r}dr\\
&<0.
\end{align*}
So we have
\begin{align*}
J(R)< J(0)=0.
\end{align*}
Consider the function
\begin{align*}
F_1(R)=R+\frac{M(R)\sinh R}{J(R)}.
\end{align*}
Then $F_1(R)$ and $F'(R)$ have the opposite signs. In addition, by direct computation we get
\begin{align*}
M(R)&=\frac{1}{2^{2m+1}(m+1)}R^{2m+2}(1+o(1)), \text{ as }R\to 0,\\
J(R)&=-\frac{1}{2^{2m+1}(m+1)}R^{2m+2}(1+o(1)), \text{ as }R\to 0.
\end{align*}
Then we may check directly that
\begin{align*}
F_1(R)=o(R), \text{ as }R\to 0.
\end{align*}
Then $F_1(0)=0$. In addition, we have
\begin{align*}
F_1'(R)=1+\frac{\left((\tanh\frac{R}{2})^{2m}\sinh^2R+M(R)\cosh R\right)J(R)-M(R)J'(R)\sinh R}{J^2(R)}.
\end{align*}
Using
\begin{align*}
(\tanh\frac{R}{2})^{2m}\sinh^2R=(2m+\cosh R)M(R)-J(R)
\end{align*}
in the above equality, we get
\begin{align*}
F_1'(R)=\frac{M(R)}{J^2(R)}\left(2(m+\cosh R)J(R)-J'(R)\sinh R\right).
\end{align*}
Define
\begin{align*}
F_2(R)&=2(m+\cosh R)J(R)-J'(R)\sinh R\\
&=(\cosh^2 R+6m\cosh R+4m^2+1)M(R)\\
&\quad -(2m+\cosh R)(\tanh\frac{R}{2})^{2m}\sinh^2R.
\end{align*}
Then
\begin{align*}
F_2'(R)&=2\sinh R\left((\cosh R+3m)M(R)-(\tanh\frac{R}{2})^{2m}\sinh^2R\right).
\end{align*}
Set
\begin{align*}
F_3(R)=(\cosh R+3m)M(R)-(\tanh\frac{R}{2})^{2m}\sinh^2R.
\end{align*}
We have
\begin{align*}
F_3'(R)&=\sinh R\left(M(R)+m(\tanh\frac{R}{2})^{2m}-(\tanh\frac{R}{2})^{2m}\cosh R\right)\\
                    &=m\sinh R\left((\tanh\frac{R}{2})^{2m}-2\int_0^R\frac{(\tanh\frac{r}{2})^{2m}}{\tanh r}dr \right).
\end{align*}
Let
\begin{align*}
F_4(R)=(\tanh\frac{R}{2})^{2m}-2\int_0^R\frac{(\tanh\frac{r}{2})^{2m}}{\tanh r}dr.
\end{align*}
Then we have
\begin{align*}
F_4'(R)=\frac{2(\tanh\frac{R}{2})^{2m}}{\sinh R}(m-\cosh R).
\end{align*}
For $m=1$, we notice that $F_4'(R)< 0$. Also, note that
\begin{align*}
F_i(0)=0,\quad 1\leq i\leq 4.
\end{align*}
Then $F_1(R)<0$. It follows that $F'(R)>0$. Therefore, we know that $\eta_{(1)} R$ is strictly increasing.

For $m\geq 2$, there exists a unique $R_m^{\star}>0$, such that $F_4(R)$ is strictly increasing on $(0,R_m^{\star})$, and strictly decreasing on $(R_m^\star,+\infty)$. Now using L'H\^{o}pital's rule and the expression for $J'(R)$ we may verify that as $R\to +\infty$,
\begin{align*}
M(R)=\cosh R(1+o(1)), \quad J(R)=-2m R\cosh R(1+o(1)).
\end{align*}
Then using them and the expressions for $F'_i(R)$ ($1\leq i\leq 4$), by L'H\^{o}pital's rule we may check in the following order that  as $R\to +\infty$,
\begin{align*}
F_4(R)&=-2R(1+o(1)),\\
F_3(R)&=-2m R \cosh R (1+o(1)),\\
F_2(R)&=-2m R \cosh^2R (1+o(1)),\\
F_1(R)&=-\frac{1}{2m}\frac{\cosh R}{R}(1+o(1)).
\end{align*}
Then we have
\begin{align*}
F_i(+\infty)=-\infty,\quad 1\leq i\leq 4.
\end{align*}
Also, note
\begin{align*}
F_i(0)=0,\quad 1\leq i\leq 4.
\end{align*}
We know that there exists an $R_m>0$, such that $\eta_{(m)} R$ is strictly decreasing on $(0,R_m)$, and strictly increasing on $(R_m,+\infty)$.

(2) Define
\begin{align*}
G(R)&=\eta_{(m)}\sinh R=\frac{(\tanh\frac{R}{2})^{2m}\sinh^2R}{M(R)}.
\end{align*}
We have
\begin{align*}
G'(R)=(\tanh\frac{R}{2})^{2m}\sinh R\cdot\frac{(2m+2\cosh R)M(R)-(\tanh\frac{R}{2})^{2m}\sinh^2R}{M^2(R)}.
\end{align*}
Set
\begin{align*}G_1(R)=M(R)-\frac{(\tanh\frac{R}{2})^{2m}\sinh^2R}{2m+2\cosh R}.\end{align*}
We know that
\begin{align*}
G_1'(R)&=\frac{(\tanh\frac{R}{2})^{2m}\sinh^3 R}{2(m+\cosh R)^2}> 0.
\end{align*}
In addition, since $G_1(0)=0$, we have $G_1(R)>0$. Then $G'(R)> 0$. It follows that $G(R)$ is strictly increasing on $(0,+\infty)$.

(3) Let
\begin{align*}P(R)&=\eta_{(m)}\tanh\frac{R}{2}=\frac{(\tanh\frac{R}{2})^{2m+1}\sinh R}{M(R)}.
\end{align*}
We have
\begin{align*}
P'(R)=\frac{(\tanh\frac{R}{2})^{2m+1}}{M^2(R)}\cdot\left((2m+1+\cosh R)M(R)-(\tanh\frac{R}{2})^{2m}\sinh^2R\right).
\end{align*}
Set
\begin{align*}
P_1(R)=(2m+1+\cosh R)M(R)-(\tanh\frac{R}{2})^{2m}\sinh^2R.
\end{align*}
Note that $P_1(0)=0$. Furthermore, we have
\begin{align*}
P_1'(R)=\sinh R\left(M(R)+(1-\cosh R)(\tanh\frac{R}{2})^{2m}\right).
\end{align*}
Let
\begin{align*}
P_2(R)=M(R)+(1-\cosh R)(\tanh\frac{R}{2})^{2m}.
\end{align*}
Then $P_2(0)=0$. In addition, we have
\begin{align*}
P_2'(R)&=(1-\cosh R)\cdot\frac{d(\tanh\frac{R}{2})^{2m}}{dR}<0.
\end{align*}
So we have $P_2(R)< 0$. Then $P_1(R)< 0$ and $P'(R)< 0$. Therefore, we know $\eta_{(m)}\tanh(R/2)$ is strictly decreasing on $(0,+\infty)$.

(4) Define
\begin{align*}
Q(R)&=\eta_{(m)}\sinh\frac{R}{2}=\frac{(\tanh\frac{R}{2})^{2m}\sinh R\sinh\frac{R}{2}}{M(R)}.
\end{align*}
Then we have
\begin{align*}
Q'(R)&=\frac{(\tanh\frac{R}{2})^{2m}\sinh\frac{R}{2}}{M^2(R)}\\
&\times  \left((2m+\cosh R+\cosh^2\frac{R}{2})M(R)-(\tanh\frac{R}{2})^{2m}\sinh^2R\right).
\end{align*}
Let
\begin{align*}
Q_1(R)=M(R)-\frac{2(\tanh\frac{R}{2})^{2m}\sinh^2R}{L(R)},
\end{align*}
where
\begin{align*}
L(R)=3\cosh R+4m+1.
\end{align*}
So we have $Q_1(0)=0$. In addition, we have
\begin{align*}
Q_1'(R)&=(\tanh\frac{R}{2})^{2m}\sinh R-\frac{4(m+\cosh R)(\tanh\frac{R}{2})^{2m}\sinh R}{L(R)}\\
&\quad +\frac{6(\tanh\frac{R}{2})^{2m}\sinh^3R}{L^2(R)}\\
&=\frac{(\tanh\frac{R}{2})^{2m}\sinh R}{L^2(R)}\cdot(\cosh R-1)(3\cosh R-4m+5).
\end{align*}

In the case $m\leq 2$, we have
$$3\cosh R-4m+5> 0,$$
 which implies that $Q_1(R)> 0$. Then $Q'(R)> 0$. It follows that $\eta_{(m)}\sinh(R/2)$ is strictly increasing on $(0,+\infty)$.

In the case $m\geq 3$, let $R_m^{\diamond}$ be the unique root of the equation
\begin{align*} 3\cosh R-4m+5=0.\end{align*}
Then we see that $Q_1(R)$ is strictly decreasing on $(0,R_m^{\diamond})$, and strictly increasing on $(R_m^{\diamond},+\infty)$.

Next, note that as $R\to +\infty$, we have
\begin{align*}
&M(R)=\sinh R(1+o(1)),\\
&(\tanh \frac{R}{2})^{2m}=1+o(1),\\
&L(R)=3\sinh R(1+o(1)).
\end{align*}
Therefore we get
\begin{align*}
Q_1(R)&=M(R)-\frac{2(\tanh\frac{R}{2})^{2m}\sinh^2R}{L(R)}\\
&=\frac{1}{3}\sinh R(1+o(1)), \text{ as }R\to +\infty,
\end{align*}
and $Q_1(R)>0$ for sufficiently large $R$.

Consequently, there exists a unique $\bar{R}_m>0$ such that $Q_1(\bar{R}_m)=0$. Then $Q(R)$ is strictly decreasing on $(0,\bar{R
}_m)$, and strictly increasing on $(\bar{R}_m,+\infty)$. Since $\eta_{(m)}\sinh(R/2)$ shares the same monotonicity with $Q(R)$, we finish the proof.
\end{proof}

Last we give the proof of Theorem~\ref{thm5}.
\begin{proof}[Proof of Theorem~\ref{thm5}]

(2) Note that
\begin{align*}\xi_{(m)}\sinh^3 R=m^2\eta_{(m)}\sinh R.\end{align*}
Then (2) holds because we have proved in Theorem~\ref{thm6} (2) that $\eta_{(m)}\sinh R$ is strictly increasing on $(0,+\infty)$.

(4) Let
\begin{align*}
q(R)&=\frac{2\xi_{(m)}\sinh^3\frac{R}{2}}{m^2}=\frac{(\tanh\frac{R}{2})^{2m+1}\sinh\frac{R}{2}}{M(R)}.
\end{align*}
We have
\begin{align*}
q'(R)=\frac{(\tanh\frac{R}{2})^{2m+1}}{2M^2(R)\cosh\frac{R}{2}}\cdot\left((2m+1+\cosh^2\frac{R}{2})M(R)-(\tanh\frac{R}{2})^{2m}\sinh^2R\right).
\end{align*}
Set
\begin{align*}
q_1(R)=M(R)-\frac{2(\tanh\frac{R}{2})^{2m}\sinh^2R}{K(R)},
\end{align*}
where
\begin{align*}
K(R)=\cosh R+4m+3.
\end{align*}
We have $q_1(0)=0$. Note
\begin{align*}
&q_1'(R)=(\tanh\frac{R}{2})^{2m}\sinh R\\
&-\frac{4m(\tanh\frac{R}{2})^{2m}\sinh R+4(\tanh\frac{R}{2})^{2m}\sinh R\cosh R}{K(R)} +\frac{2(\tanh\frac{R}{2})^{2m}\sinh^3R}{K^2(R)}\\
&\qquad \  =\frac{(\tanh\frac{R}{2})^{2m}\sinh R}{K^2(R)}\cdot(1-\cosh R)(\cosh R+12m+7)\\
&\quad\quad  \ < 0.
\end{align*}
Thus we have $q_1(R)< 0$. Then $q'(R)< 0$. Therefore, we see that $\xi_{(m)}\sinh^3(R/2)$ is strictly decreasing on $(0,+\infty)$.

(1) Note that the function
\begin{align*}R\mapsto\frac{R}{\sinh \frac{R}{2}}\end{align*}
is strictly decreasing on $(0,+\infty)$. In Theorem~\ref{thm5} (4) we have proved that $\xi_{(m)} \sinh^3(R/2)$ is strictly decreasing on $(0,+\infty)$. Also, we have
\begin{align*}
\xi_{(m)}R^3=\xi_{(m)}\sinh^3\frac{R}{2}\cdot \frac{R^3}{\sinh^3\frac{R}{2}}.
\end{align*}
Then the conclusion (1) holds.

(3) It is not hard to prove that the function
\begin{align*}
R\mapsto\frac{\tanh\frac{R}{2}}{\sinh \frac{R}{2}}
\end{align*}
is strictly decreasing on $(0,+\infty)$. Next note that
\begin{align*}
\xi_{(m)}\tanh^3\frac{R}{2}=\xi_{(m)}\sinh^3\frac{R}{2}\cdot\frac{\tanh^3\frac{R}{2}}{\sinh^3\frac{R}{2}}.
\end{align*}
We finish the proof of (3), since we have shown in Theorem~\ref{thm5} (4) that $\xi_{(m)} \sinh^3(R/2)$ is strictly decreasing on $(0,+\infty)$.

\end{proof}

\section{Proofs of monotonicity of Steklov eigenvalues on geodesic disks with varying curvature}\label{sec5}

In this section, first we give the proof of Corollary~\ref{thm7}.
\begin{proof}[Proof of Corollary~\ref{thm7}]
(1) First note that the radius $R$ of the geodesic disk of fixed area $A>0$ and curvature $K$ satisfies
\begin{align*}
A=\int_0^R 2\pi h_K(r)dr=
\begin{cases}
4\pi \dfrac{\sin^2 (\frac{\sqrt{K}R}{2})}{K},\quad &\text{if }K>0,\\
\pi R^2,\quad &\text{if }K=0,\\
-4\pi \dfrac{\sinh^2 (\frac{\sqrt{-K}R}{2})}{K},\quad &\text{if }K<0.
\end{cases}
\end{align*}
Hence $R$ can be viewed as a function $R=R(K)$ of the curvature $K$ when $A>0$ is fixed. Then we may consider
\begin{equation*}
\sigma_{(m)}(K;A)=\frac{m}{h_K(R)}=\frac{m}{h_K(R(K))}
\end{equation*}
as a function of $K$ and study its monotonicity with respect to $K\in (-\infty,4\pi/A)$. Next we divide the argument into two cases.

(1.1) The case $K>0$. According to scaling properties for the Steklov problem, we know that
\begin{equation*}
\sigma_{(m)}(K;A)=\sqrt{K}\sigma_{(m)}(1; KA)=\sqrt{4\pi A^{-1}}\sigma_{(m)}(\Theta)\sin\frac{\Theta}{2},
\end{equation*}
where the geodesic radius of the geodesic disk of area $KA$ in $\SS^2$ of constant curvature $+1$,
\begin{align*}
\Theta=2\arcsin \sqrt{\frac{KA}{4\pi}}
\end{align*}
is strictly increasing with respect to $K\in (0,4\pi/A)$.

(1.2) The case $K<0$. According to scaling properties for the Steklov problem, we know that
\begin{equation*}
\sigma_{(m)}(K;A)=\sqrt{-K}\sigma_{(m)}(-1;-KA)=\sqrt{4\pi A^{-1}}\sigma_{(m)}(\Theta)\sinh\frac{\Theta}{2},
\end{equation*}
where the geodesic radius of the geodesic disk of area $-KA$ in $\H^2$ of constant curvature $-1$,
\begin{align*}
\Theta=2\mathrm{arsinh} \sqrt{\frac{-KA}{4\pi}}
\end{align*}
is strictly decreasing with respect to $K\in (-\infty,0)$.

Recall $\sigma_{(m)}(\Theta)=m/\sin \Theta$ or $m/\sinh \Theta$. It is straightforward to check that the conclusion holds. So we complete the proof of (1).

(2) When the radius $\rho>0$ of the geodesic disk of curvature $K$ is fixed, the $m$th eigenvalue $\sigma_{(m)}(K;\rho)$ is given by
\begin{equation*}
\sigma_{(m)}(K;\rho)=\frac{m}{h_K(\rho)},
\end{equation*}
and we may study its monotonicity with respect to $K\in (-\infty,(\pi/\rho)^2)$. Again we divide the argument into two cases.

(2.1) The case $K>0$.
We have
\begin{equation*}
\sigma_{(m)}(K;\rho)=\sqrt{K}\sigma_{(m)}(1;\sqrt{K}\rho)=\rho^{-1}\sigma_{(m)}(\Theta)\Theta,
\end{equation*}
where $\Theta=\sqrt{K}\rho$ is strictly increasing on $(0,4\pi/A)$ with respect to $K$.

(2.2) The case $K<0$. We have
\begin{equation*}
\sigma_{(m)}(K;\rho)=\sqrt{-K}\sigma _{(m)}(-1;\sqrt{-K}\rho)=\rho^{-1}\sigma_{(m)}(\Theta)\Theta,
\end{equation*}
where $\Theta=\sqrt{-K}\rho$ is strictly decreasing on $(-\infty,0)$ with respect to $K$.

Recall $\sigma_{(m)}(\Theta)=m/\sin \Theta$ or $m/\sinh \Theta$. We can finish the proof of~(2).
\end{proof}

Next we give the proof of Corollary~\ref{thm8}.
\begin{proof}[Proof of Corollary~\ref{thm8}]
$(1)$ We divide the argument into two cases.

$(1.1)$ The case $K>0$. According to scaling properties for the Steklov problem, we know that
\begin{equation*}
\xi_{(m)}(K;A)=K^{\frac{3}{2}}\xi_{(m)}(1;KA)=(4\pi A^{-1})^{\frac{3}{2}}\xi_{(m)}(\Theta)\sin^3\frac{\Theta}{2},
\end{equation*}
where $\Theta$ is strictly increasing with respect to $K\in (0,4\pi /A)$ as explained in the proof of Corollary~\ref{thm7}.

$(1.2)$ The case $K<0$. According to scaling properties for the Steklov problem, we know that
\begin{equation*}
\xi_{(m)}(K;A)=(-K)^{\frac{3}{2}}\xi_{(m)}(-1;-KA)=(4\pi A^{-1})^{\frac{3}{2}}\xi_{(m)}(\Theta)\sinh^3\frac{\Theta}{2},
\end{equation*}
where $\Theta$ is strictly decreasing with respect to $K\in (-\infty,0)$ as explained in the proof of Corollary~\ref{thm7}.

Now according to Theorem ~\ref{thm2} (4) and Theorem ~\ref{thm5} (4) we complete the proof of (1).

$(2)$ We divide the argument into two cases.

$(2.1)$ The case $K>0$. According to scaling properties for the Steklov problem, we have
\begin{equation*}
\xi_{(m)}(K;\rho)=K^{\frac{3}{2}}\xi_{(m)}(1;\sqrt{K}\rho)=\rho^{-3}\xi_{(m)}(\Theta)\Theta^3,
\end{equation*}
where $\Theta=\sqrt{K}\rho$ is strictly increasing with respect to $K$.

$(2.2)$ The case $K<0$. We have, according to scaling properties for the Steklov problem,
\begin{equation*}
\xi_{(m)}(K;\rho)=(-K)^{\frac{3}{2}}\xi_{(m)}(1;\sqrt{-K}\rho)=\rho^{-3}\xi_{(m)}(\Theta)\Theta^3,
\end{equation*}
where $\Theta=\sqrt{-K}\rho$ is strictly decreasing with respect to $K$.

Now according to Theorem ~\ref{thm2} (1) and Theorem ~\ref{thm5} (1) we complete the proof of (2).
 \end{proof}

Last we give the proof of Corollary~\ref{thm9}.
\begin{proof}[Proof of Corollary~\ref{thm9}]
$(1)$ We divide the argument into two cases.

$(1.1)$ The case $K>0$. We know that
\begin{equation*}
\eta_{(m)}(K;A)=\sqrt{K}\eta_{(m)}(1;KA)=\sqrt{4\pi A^{-1}}\eta_{(m)}(\Theta)\sin\frac{\Theta}{2},
\end{equation*}
where $\Theta$ is strictly increasing with respect to $K\in (0,4\pi /A)$ as explained in the proof of Corollary~\ref{thm7}.

$(1.2)$ The case $K<0$. We know that
\begin{equation*}
\eta_{(m)}(K;A)=\sqrt{-K}\eta_{(m)}(-1; -KA)=\sqrt{4\pi A^{-1}}\eta_{(m)}(\Theta)\sinh\frac{\Theta}{2},
\end{equation*}
where $\Theta$ is strictly decreasing with respect to $K\in (-\infty,0)$ as explained in the proof of Corollary~\ref{thm7}.

Now according to Theorem ~\ref{thm3} (4) and Theorem ~\ref{thm6} (4) we complete the proof of (1).

$(2)$ We divide the argument into two cases.

$(2.1)$ The case $K>0$. We have
\begin{equation*}
\eta_{(m)}(K;\rho)=\sqrt{K}\eta_{(m)}(1;\sqrt{K}\rho)=\rho^{-1}\eta_{(m)}(\Theta)\Theta,
\end{equation*}
where $\Theta=\sqrt{K}\rho$ is strictly increasing on $(0,4\pi/A)$ with respect to $K$.

$(2.2)$ The case $K<0$. We have
\begin{equation*}
\eta_{(m)}(K;\rho)=\sqrt{-K}\eta_{(m)}(-1;\sqrt{-K}\rho)=\rho^{-1}\eta_{(m)}(\Theta)\Theta,
\end{equation*}
where $\Theta=\sqrt{-K}\rho$ is strictly decreasing on $(-\infty,0)$ with respect to $K$.

Now according to Theorem ~\ref{thm3} (1) and Theorem ~\ref{thm6} (1) we complete the proof of (2).
\end{proof}

\section{Proofs of scaling inequalities for $n\geq 3$}\label{sec6}
First we give the proof of Theorem~\ref{thm 10}.
\begin{proof}[Proof of Theorem~\ref{thm 10}]
Considering the ODE ~\eqref{2nd}, by Proposition~\ref{prop1} we know that the $m$th eigenvalue $\sigma_{(m)}$ without counting multiplicity is given by
\begin{equation}\label{6.1}
 \sigma_{(m)}=\frac{\psi'(R)}{\psi(R)}.
 \end{equation}
Multiplying the equation \eqref{2nd} by $h^{n-1}$ and integrating from $0$ to $R$, we obtain
\begin{equation}\label{6.2}
h^{n-1}\psi'=\tau_m\int_0^Rh^{n-3}\psi dr.
 \end{equation}
Here and below for simplicity we write $h$ for $h(R)$ or $h(r)$ whenever no confusion arises, and the same convention applies to $\psi$. Now combining ~\eqref{6.1} and ~\eqref{6.2} to eliminate $\psi'$, we obtain
\begin{equation*}
\sigma_{(m)}h=\frac{\psi'}{\psi}h=\frac{\tau_m\int_0^Rh^{n-3}\psi dr}{h^{n-2}\psi }.
 \end{equation*}
Let $$F(R)=\frac{\int_0^R h^{n-3}\psi dr}{h^{n-2}\psi}.$$
Then we have
\begin{align*}
F'(R)=\frac{h^{2n-5}\psi^2-\left((n-2)h^{n-3}h'\psi+h^{n-2}\psi'\right)\int_0^Rh^{n-3}\psi dr}{(h^{n-2}\psi)^2}.
\end{align*}
Set
\begin{align*}
G(R)=h^{n-2}\psi^2-\left((n-2)h'\psi+h\psi'\right)\int_0^Rh^{n-3}\psi dr.
\end{align*}
Then $G(0)=0$. In addition, we have
\begin{align*}
G'(R)=h^{n-2}\psi\psi'-\left((n-2)h''\psi+(n-1)h'\psi'+h\psi''\right)\int_0^Rh^{n-3}\psi dr.
\end{align*}
Using
\begin{align*}
h\psi''+(n-1)h'\psi'=\frac{\tau_m\psi}{h}
\end{align*}
and
\begin{align*}
\int_0^R h^{n-3}\psi dr=\frac{h^{n-1}\psi'}{\tau_m},
\end{align*}
we obtain
\begin{align*}
G'(R)&=-\frac{n-2}{\tau_m}h^{n-1}h''\psi \psi'.
\end{align*}

If we are in the hyperbolic space, then $h''=h$ and
\begin{align*}
G'(R)&=-\frac{n-2}{\tau_m}h^n\psi\psi'<0.
\end{align*}
So we have $G(R)< 0$. Then $F'(R)< 0$. It follows that $\sigma_{(m)}\sinh R$ is strictly decreasing on $(0,+\infty)$.

If we are in the sphere, then $h''=-h$ and we can get the conclusion similarly. Thus we complete the proof.

\end{proof}

To make the proofs of Theorem~\ref{thm 11}, Theorem~\ref{thm 14} and Theorem~\ref{thm 12} more concise, we define
\begin{equation*}
y_m=\frac{u_m'}{u_m}.
\end{equation*}
Then one can readily check that $y_m$ satisfies the Riccati differential equation
\begin{equation}\label{eq6.3}
y_m'=-y_m^2-\frac{(n-1)h'}{h}y_m+\frac{\tau_m}{h^2}.
\end{equation}
Also, we define
\begin{equation*}
z_m=hy_m.
\end{equation*}
Then $z_m$ satisfies
\begin{equation}\label{eq-z}
z_m'=-\frac{z_m^2+(n-2)h'z_m-\tau_m}{h}.
\end{equation}

Now we give the proof of Theorem ~\ref{thm 11}.
\begin{proof}[Proof of Theorem ~\ref{thm 11}]
We divide the argument into two cases.

(1) For $m\geq 1$, consider the function
\begin{align*}
F=\eta_{(m)} \sinh R=\frac{h^nu_m^2}{\int_0^R h^{n-1}u_m^2dr}.
\end{align*}
Then we have
\begin{align*}
F'=\frac{h^{n-1}u_m(nh'u_m+2hu_m')}{(\int_0^R h^{n-1}u_m^2dr)^2}\left(\int_0^R h^{n-1}u_m^2dr-\frac{h^nu_m^3}{nh'u_m+2hu_m'}\right).
\end{align*}
Next set
\begin{align*}
G&=\int_0^R h^{n-1}u_m^2dr-\frac{h^nu_m^3}{nh'u_m+2hu_m'}\\
&=\int_0^R h^{n-1}u_m^2dr-\frac{h^nu_m^2}{nh'+2z_m}.
\end{align*}
For later use, we need some computation results for $z_m$ and its derivatives at zero. By direct computation using Equations~\eqref{eq6.3} and~\eqref{eq-z}, we get
\begin{align}\label{eq-z0}
z_m(0)=m,\quad z'_m(0)=0,\quad z''_m(0)=-\frac{(n-2)m}{2m+n}.
\end{align}
Then $G(0)=0$. In addition, we have
\begin{align*}
G'=\frac{h^{n-1}u_m^2}{(nh'+2z_m)^2}\left(nh^2+2\tau_m-2(n-2)h'z_m-2z_m^2\right).
\end{align*}
Now consider the function
\begin{equation}\label{eq6.6}
H=nh^2+2\tau_m-2(n-2)h'z_m-2z_m^2.
\end{equation}
Then we deduce that
\begin{align}\label{eq6.7}
H'
&= \frac{2}{h}\Bigl(
    2z_m^{3} + 3(n-2)h' z_m^{2}+ \left((n-2)^{2}(h')^{2} - (n-2)h^{2} - 2\tau_m\right) z_m \nonumber\\
&\quad
   +n h^{2}h'- (n-2)\tau_m h'
   \Bigr).
\end{align}
Next we claim
\begin{equation*}
H>0,\quad  \forall R>0.
\end{equation*}
The proof is by contradiction. First by using \eqref{eq-z0} to get
\begin{align*}
z_m(R)=m-\frac{(n-2)m}{2(2m+n)}R^2+o(R^2),\text{ as } R\to 0,
\end{align*}
and further using
\begin{align*}
h(R)&=\sinh R=R+o(R^2),\\
h'(R)&=\cosh R=1+\frac{1}{2}R^2+o(R^2),\text{ as } R\to 0,
\end{align*}
we obtain
\begin{align*}
H=\frac{n^2+4m}{2m+n}R^2+o(R^2),\text{ as } R\to 0.
\end{align*}
It follows that there exists an $R_m>0$ such that $H>0$ on $(0,R_m)$.
Suppose the conclusion is not true. Let
\begin{align*}
R_0=\inf\{R\big|H(R)=0\}.
\end{align*}
Then $R_0>0$. Moreover, we have $H'(R_0)\leq 0$. Take $R=R_0$ in \eqref{eq6.6}. Then we obtain
\begin{equation*}
2z_m^2(R_0)=(nh^2+2\tau_m-2(n-2)h'z_m)\big|_{R=R_0}.
\end{equation*}
Substituting this expression into \eqref{eq6.7}, we have
\begin{align*}
H'(R_0)=\frac{2}{h}(2h^2z_m+\frac{n^2}{2}h^2h')\big|_{R=R_0}>0,
\end{align*}
which is a contradiction. Now note that
\begin{align*}
G(0)=H(0)=0.
\end{align*}
One can readily draw the desired conclusion for $m \geq 1$.

(2) For $m=0$, consider the function
\begin{align*}
J=\eta_{(0)}\sinh R=\frac{h^n}{\int_0^R h^{n-1}dr}.
\end{align*}
Then we have
\begin{align*}
J'=\frac{h^{n-1}}{(\int_0^Rh^{n-1}dr)^2}\left(nh'\int_0^R h^{n-1}dr-h^n\right).
\end{align*}
Set
\begin{align*}
K=nh'\int_0^R h^{n-1}dr-h^n.
\end{align*}
Notice that $K(0)=0$. Moreover, we have
\begin{align*}
K'=nh''\int_0^R h^{n-1}dr > 0.
\end{align*}
Then $K> 0.$ In addition, $J'> 0$. It follows that $\eta_{(0)}\sinh R$ increases strictly on $(0,+\infty)$. Thus we complete the proof.
\end{proof}

\section{Proofs of sharp bounds for eigenvalues on warped product manifolds}\label{sec7}

First we give the proof of Theorem ~\ref{thm 14}.

\begin{proof}[Proof of Theorem ~\ref{thm 14}]
First we claim that
\begin{align*}
z_m=hy_m
\end{align*}
is increasing on $(0,R]$. Recall from ~\eqref{eq-z} that
\begin{equation*}
z_m'=-\frac{z_m^2+(n-2)h'z_m-\tau_m}{h}.
\end{equation*}
Set
\begin{align*}
F=z_m^2+(n-2)h'z_m-\tau_m.
\end{align*}
Notice that $F(0)=0$. Then our aim is to show $F\leq 0$. Note that
\begin{align*}
F'&=2z_mz_m'+(n-2)h''z_m+(n-2)h'z_m'\\
  &=-\frac{2z_m+(n-2)h'}{h}F+(n-2)h''z_m.
\end{align*}
It follows that
\begin{align}\label{example}
F'+c(r)F\leq 0,
\end{align}
where $$c(r)=\frac{2z_m+(n-2)h'}{h}.$$
Then we see that, for sufficiently small $\epsilon>0$,
\begin{equation}\label{eq7.2}
(e^{\int_{\epsilon}^r c(t)dt}F)' \leq 0.
\end{equation}
Integrating both sides from $\epsilon$ to $R$ in \eqref{eq7.2}, we have
\begin{align*}
e^{\int_{\epsilon}^R c(t)dt}F(R)\leq F(\epsilon).
\end{align*}
Notice that
\begin{align*}
\int_{\epsilon}^Rc(t)dt\to +\infty, \text{ as } \epsilon\to 0+.
\end{align*}
We have
\begin{align*}
F(R)\leq e^{-\int_{\epsilon}^R c(t)dt}F(\epsilon)\to 0,\text{ as } \epsilon\to 0+,
\end{align*}
which leads to $F(R)\leq 0$. Hence the claim follows.

Next assume $n\geq 4$ and we prove the lower bound for $\xi_{(m)}$. Let
\begin{align*}
G=h^{n+2}(u_m')^2-m^2(n+2m)\int_0^R h^{n-1}u_m^2dr.
\end{align*}
Then $G(0)=0$. In addition, we have
\begin{align*}
G'=h^{n-1} u_m^2\left(-m^2\left(n+2m\right)+2\tau_mz_m-(n-4)h'z_m^2\right).
\end{align*}
Then we consider the function
\begin{equation}\label{eq6.12}
H=-m^2\left(n+2m\right)+2\tau_mz_m-(n-4)h'z_m^2.
\end{equation}
According to ~\eqref{eq-z} we obtain
\begin{equation}\label{eq h'}
h'=\frac{\tau_m-hz_m'-z_m^2}{(n-2)z_m}.
\end{equation}
Substituting this expression into \eqref{eq6.12}, we have
\begin{align*}
H=-m^2(n+2m)+\frac{n}{n-2}\tau_mz_m+\frac{n-4}{n-2}(z_m^3+hz_mz_m').
\end{align*}
Recall $n\geq 4$, and (see Prop.~12 in \cite{Xio21})
\begin{align*}
z_m\geq m, \quad z_m'\geq 0.
\end{align*}
It follows that
\begin{equation}\label{eq6.13}
H\geq H\big|_{z_m=m,z_m'=0}=0.
\end{equation}
Then $G'\geq 0$. Moreover, we have $G\geq 0$. It follows that
\begin{align*}
\xi_{(m)} \geq m^2(n+2m)\dfrac{1}{h^3(R)}.
\end{align*}
In addition, when $H=0$, from ~\eqref{eq6.13} we obtain
\begin{align*}
z_m\equiv m,\quad z_m'\equiv 0.
\end{align*}
Moreover, from ~\eqref{eq h'} we have $h'\equiv1$. Then $h(r)\equiv r$.

Now we note that the argument above can be used to prove the upper bound in ~\eqref{ineq-xi2} when $n=2$. In fact for $n=2$, we have (see Thm.~2 in \cite{Xio21})
\begin{align*}
z_m=\frac{hu_m'}{u_m}=h\sigma_{(m)}=m.
\end{align*}
Then ~\eqref{eq6.12} becomes
\begin{equation}\label{eqH}
H=2m^2(h'-1)\leq0.
\end{equation}
Thus $G'\leq 0$, leading to $G\leq 0$. It follows that
\begin{align*}
\xi_{(m)} \leq m^2(2+2m)\dfrac{1}{h^3(R)}.
\end{align*}
In addition, when $H=0$, from ~\eqref{eqH} we have $h'\equiv1$. Then $h(r)\equiv r$.

Next we consider the case $n\leq 3$ and prove the lower bound for $\xi_{(m)}$. Set
\begin{align*}
I=\frac{h^{n+2}(u_m')^2}{h'}-m^2(n+2m)\int_0^R h^{n-1}u_m^2 dr.
\end{align*}
Then $I(0)=0$. In addition, we have
\begin{align*}
I'&=h^{n-1} u_m^2 \left(-\left(n-4\right)z_m^2+\frac{2\tau_mz_m}{h'}-m^2\left(n+2m\right)\right)-\frac{h^{n+2}(u_m')^2}{(h')^2}h''\\
  &\geq h^{n-1} u_m^2\left(-\left(n-4\right)z_m^2+\frac{2\tau_mz_m}{h'}-m^2\left(n+2m\right)\right)\\
  & \geq h^{n-1} u_m^2\left(-\left(n-4\right)m^2+2m\tau_m-m^2\left(n+2m\right)\right)\\
  &= 0.
\end{align*}
Then $I\geq 0$, which implies that
\begin{align*}
\xi_{(m)} \geq m^2(n+2m)\dfrac{h'(R)}{h^3(R)}.
\end{align*}
In addition, when $I'=0$, we can check that $h''\equiv 0$ and $h'\equiv 1$. It follows that $h(r)\equiv r.$ Hence the proof is completed.
\end{proof}
\begin{rem}
In the proof of Theorem~\ref{thm 14}, we showed that $z_m$ is increasing on $(0,R]$. Moreover, if we replace the condition $h''\leq0$ and $h'\leq 1$ by $h''\geq 0$ and $h'\geq 1$, then we may prove that $z_m$ decreases on $(0,R]$. Note $z_m=hu_m'/u_m=h\sigma_{(m)}$. Hence the above conclusions contain Theorem ~\ref{thm 10} (where $h(r)=\sin r$ or $h(r)=\sinh r$) and the argument gives an alternative proof for it.
\end{rem}
Then we give the proof of Theorem~\ref{thm 12}.
\begin{proof}[Proof of Theorem~\ref{thm 12}]
We divide the argument into two cases.

(1) Let $m\geq 1$. First we prove the upper bound in \eqref{ineq-eta}. Set
\begin{align*}
F(R)=h^nu_m^2-(n+2m)\int_0^R h^{n-1}u_m^2dr.
\end{align*}
Note that $F(0)=0$. In addition, we have
\begin{align*}
F'(R)=nh^{n-1}h'u_m^2+2h^nu_mu_m'-(n+2m)h^{n-1}u_m^2.
\end{align*}
Let
\begin{align*}
G(R)=nh'u_m+2hu_m'-(n+2m)u_m.
\end{align*}
Notice that $G(0)=0$. Moreover, we have
\begin{align*}
G'(R) &= nh''u_m + nh'u_m' + 2h'u_m' + 2hu_m'' - (n+2m)u_m' \\[3pt]
& \le \left(- (n-4)h' - (n+2m)\right)u_m' + \frac{2\tau_m}{h}u_m \\[3pt]
&= -\left((n-4)h' + (n+2m)\right)\frac{G}{2h} \\[2pt]
&\quad -\left((n+2m) - nh'\right)\left((n+2m) + (n-4)h'\right)\frac{u_m}{2h}
   + \frac{2\tau_m}{h}u_m.
\end{align*}
Now we claim
\begin{align*}
I:=4\tau_m-\left((n+2m) - nh'\right)\left((n+2m) + (n-4)h'\right)\leq0.
\end{align*}
Note that
\begin{align*}
I=-8m-n^2+4(n+2m)h'+n(n-4)(h')^2.
\end{align*}
Whether $n\geq 4$ or $2\leq n\leq 3$, we have
\begin{equation}\label{eq I}
I \le I\big|_{h'=1}=0.
\end{equation}
By the claim,
\begin{equation}\label{ineq G}
G'+c_1(r)G\leq 0,
\end{equation}
where $$c_1(r)=\frac{(n-4)h'+(n+2m)}{2h}.$$
Using the same argument as in the treatment of ~\eqref{example}, we obtain from ~\eqref{ineq G} that $G(R)\leq 0$. Then $F(R)\leq 0$. It follows that
\begin{align*}
\eta_{(m)} \leq (n+2m)\frac{1}{h}.
\end{align*}
In addition, when $G=0$, from ~\eqref{eq I} we have $h' \equiv 1$. Then $h(r) \equiv r$.

Next we prove the lower bound in \eqref{ineq-eta}. Set
\begin{align*}
H(R)=\frac{h^nu_m^2}{h'}-(n+2m)\int_0^R h^{n-1}u_m^2 dr.
\end{align*}
Note that $H(0)=0$. In addition, we have
\begin{align*}
H'(R)&=\frac{nh^{n-1}h'u_m^2+2h^nu_mu_m'}{h'}-\frac{h^nu_m^2}{(h')^2}h''-(n+2m)h^{n-1}u_m^2\\
        &\geq -2mh^{n-1}u_m^2+\frac{2h^nu_mu_m'}{h'}.
\end{align*}
By Prop.~12 in \cite{Xio21} we know that
\begin{equation}\label{=}
\frac{u_m'}{u_m}\geq m\frac{h'}{h}.
\end{equation}
It follows that
\begin{align*}
H'(R)\geq -2mh^{n-1}u_m^2+2mh^{n-1}u_m^2=0.
\end{align*}
Then $H(R)\geq 0$, which implies that
\begin{align*}
\eta_{(m)}\geq (n+2m)\frac{h'}{h}.
\end{align*}
Moreover, when $H'=0$, from ~\eqref{=} we obtain $z_m=mh'$. Since
\begin{align*}
z_m\geq m,\quad h'\leq 1,
\end{align*}
it follows that $h'\equiv 1$. Then $h(r)\equiv r$.

Last we prove \eqref{ineq-eta-ratio}. Set
\begin{align*}
K(R)=(n+2m)\frac{u_{m+1}^2}{u_m^2}\int_0^R h^{n-1}u_{m}^2 dr-(n+2m+2)\int_0^R h^{n-1}u_{m+1}^2 dr.
\end{align*}
Note that $K(0)=0$. In addition, we have
\begin{align*}
K'(R)
&= (n+2m)
   \frac{
      2u_m^2 u_{m+1}' u_{m+1}- 2u_m u_{m+1}^2 u_m'
   }{u_m^4}
   \int_0^R h^{n-1}u_m^2\,dr  \\[4pt]
&\quad -\, 2h^{n-1}u_{m+1}^2.
\end{align*}
Let
\begin{align*}
L(R)&=(n+2m)\int_0^R h^{n-1}u_m^2 dr-\frac{h^{n-1}u_m^3u_{m+1}}{u_mu_{m+1}'-u_{m+1}u_m'}\\
&=(n+2m)\int_0^R h^{n-1}u_m^2 dr-\frac{h^{n-1}u_m^2}{\Delta y},
\end{align*}
where $\Delta y:=y_{m+1}-y_m$.

Since
\begin{align*}
\Delta y=y_{m+1}-y_m=\sigma_{m+1}-\sigma_m,
\end{align*}
by Prop.~12 in \cite{Xio21} we know that
\begin{align*}
\Delta y \geq \frac{1}{h} >0.
\end{align*}
Then $L(R)$ and $K'(R)$ share the same sign.

Next notice that
\begin{align*}
h\Delta y|_{R=0}&=z_{m+1}(0)-z_{m}(0)=1,
\end{align*}
implying $L(0)=0$. Then we may check
\begin{align*}
L'(R)=\frac{h^{n-1}u_m^2}{\Delta y}\left((n+2m-1)\Delta y+\frac{n+2m-1}{h^2 \Delta y}-\frac{2(n-1)h'}{h}-4y_m\right).
\end{align*}
Define
\begin{align*}
M(R)=(n+2m-1)\Delta y+\frac{n+2m-1}{h^2 \Delta y}-\frac{2(n-1)h'}{h}-4y_m.
\end{align*}
By the arithmetic--geometric mean inequality we have
\begin{align*}
M(R)\geq \frac{2(2m+n-1)}{h}-\frac{2(n-1)h'}{h}-4y_m.
\end{align*}
Set
\begin{align*}
N(R)=(2m+n-1)u_m-(n-1)h'u_m-2hu_m'.
\end{align*}
Since
\begin{align*}
h''\leq 0, \quad u_m''=\frac{\tau_m}{h^2}u_m-\frac{(n-1)h'}{h}u_m',
\end{align*}
we have
\begin{align*}
N'(R)
&= (2m+n-1)u_m'-(n-1)h''u_m-(n-1)h'u_m'-2h'u_m'-2hu_m'' \\[3pt]
&\geq \left((2m+n-1)+(n-3)h'\right)u_m' - \frac{2\tau_m}{h}u_m \\[3pt]
&= \left(
    (2m+n-1)^2
    - 2(2m+n-1)h'
    - (n-1)(n-3)(h')^2 \right. \\
&\left.
    \quad - 4m(m+n-2)
  \right)\frac{u_m}{2h}
  - \frac{(2m+n-1)+(n-3)h'}{2h}N.
\end{align*}
Let
\begin{align*}
J=&(2m+n-1)^2-2(2m+n-1)h'\\
&-(n-1)(n-3)(h')^2-4m(m+n-2).
\end{align*}
Then we have
\begin{equation}\label{eqJ}
J \geq J\big|_{h'=1}=0.
\end{equation}
It follows that
\begin{equation}\label{example'}
N'+c_2(r)N\geq0,
\end{equation}
where
\begin{align*}
c_2(r)=\frac{(2m+n-1)+(n-3)h'}{2h}.
\end{align*}
Using the same argument as in the treatment of ~\eqref{example}, we obtain from ~\eqref{example'} that $N(R)\geq 0$. Then $M(R)\geq 0$. Moreover, we have $L(R)\geq 0$. It follows that $ K(R)\geq 0$. Therefore, we have
\begin{align*}
\frac{\eta_{(m+1)}}{\eta_{(m)}} \geq \frac{n+2m+2}{n+2m}.
\end{align*}
In addition, when $ N'=0$, from ~\eqref{eqJ} we have $h'\equiv 1$. Then $h(r)\equiv r$. So we complete the proof for $m \geq 1$.

(2) For $m=0$, first consider the function
\begin{align*}
P=h^n-nh'\int_0^R h^{n-1}dr.
\end{align*}
Then $P(0)=0$. In addition, we have
\begin{align*}
P'=-nh''\int_0^R h^{n-1}dr \geq 0.
\end{align*}
So we have $P\geq0$. It follows that
\begin{align*}
\eta_{(0)}\geq \frac{nh'}{h}.
\end{align*}
Moreover, when $P'=0$, we have $h''\equiv 0$. Then $h(r)\equiv r$.

Next, we set
\begin{align*}
Q=h^n-n\int_0^R h^{n-1}dr.
\end{align*}
Then $Q(0)=0$. In addition, we have
\begin{align*}
Q'=nh^{n-1}(h'-1)\leq 0.
\end{align*}
So we have $Q\leq 0$. It follows that
\begin{align*}
\eta_{(0)} \leq \frac{n}{h}.
\end{align*}
Moreover, when $Q'=0$, we have $h'\equiv 1$. It follows that $h(r)\equiv r$.

Last set
\begin{align*}
U=nu_1^2\int_0^R h^{n-1}dr-(n+2)\int_0^R h^{n-1}u_1^2dr.
\end{align*}
Then $U(0)=0$. In addition, we have
\begin{align*}
U'=2nu_1u_1'\int_0^R h^{n-1}dr-2h^{n-1}u_1^2.
\end{align*}
Now consider the function
\begin{align*}
V=n\int_0^R h^{n-1}dr-\frac{h^{n-1}u_1}{u_1'}=n\int_0^R h^{n-1}dr-\frac{h^n}{z_1}.
\end{align*}
Then $V(0)=0$. Moreover, from ~\eqref{eq-z} we have
\begin{align*}
V'&=nh^{n-1}-\frac{nh^{n-1}h'z_1-h^nz_1'}{z_1^2}\\
&=\frac{(n-1)h^{n-1}}{z_1^2}(z_1^2-2h'z_1+1)\\
&\geq \frac{(n-1)h^{n-1}}{z_1^2}(z_1^2-2z_1+1)\\
&=\frac{(n-1)h^{n-1}}{z_1^2}(z_1-1)^2\\
&\geq 0,
\end{align*}
where the first inequality is due to $h'\leq 1$. So we have $V\geq 0$. Then $U\geq 0$. It follows that
\begin{align*}
\frac{\eta_{(1)}}{\eta_{(0)}}\geq\frac{n+2}{n}.
\end{align*}
In addition, when $V'=0$, we have $h'\equiv 1$. It follows that $h(r)\equiv r$. So we complete the proof.
\end{proof}

\section{Further problems}\label{sec8}
Finally, we would like to pose several open problems that remain unresolved in this work.
\begin{ques}
\begin{enumerate}
\item In the hyperbolic space $\H^n$ $(n \geq 3)$, for the fourth-order Steklov eigenvalue $\eta_{(m)}$, can we get the monotonicity of the quantities $\eta_{(m)} R$, $\eta_{(m)}\tanh(R/2)$ and $\eta_{(m)}\sinh(R/2)$?
\item Under the assumptions as in Theorem~\ref{thm 14}, for $n=3$, is it possible to determine whether the inequality $$\xi_{(m)}\geq m^2(n+2m)\frac{1}{h^3}$$ or $$\xi_{(m)}\leq m^2(n+2m)\frac{1}{h^3}$$ holds?
\item Under the assumptions as in Theorem~\ref{thm 14}, can we establish the ratio estimate for $\xi_{(m+1)}/\xi_{(m)}$?
\end{enumerate}
\end{ques}

\bibliographystyle{Plain}

\end{document}